\documentclass[11pt,reqno,a4paper]{article}
 \usepackage{multicol,a4wide}
 \usepackage{amsmath}
\usepackage[margin=20mm]{geometry}
 \usepackage{color}
 \usepackage{cmap,mathtools}
 \usepackage{enumerate}
 \usepackage{cite}
 \usepackage{hyperref}
 \hypersetup{linkcolor=blue, colorlinks=true ,citecolor = red}

\newcommand{\dx}{{\mathrm{d}x}}

\newcommand{\dy}{{\mathrm{d}y}}
\newcommand{\dxy}{{\mathrm{d}x\mathrm{d}y}}
\newcommand{\D}{{\mathcal{D}}}
 \usepackage{amsmath,amssymb,amsthm,amsfonts,epsfig,enumerate}
 \usepackage{mathtools}
 \usepackage[hyperpageref]{backref}
 \usepackage{bigints}
 \usepackage{esint}
\newtheorem{definition}{Definition}[section]
\newtheorem{theorem}[definition]{Theorem}
\newtheorem{example}[definition]{Example}
\newtheorem{corollary}[definition]{Corollary}
\newtheorem{remark}[definition]{Remark}
\newtheorem{proposition}[definition]{Proposition}

\newtheorem{lemma}[definition]{Lemma}

\numberwithin{equation}{section}

\makeatletter

\newcommand{\Ga} {\Gamma}

\newcommand{\Om} {\Omega}

\newcommand{\no} {\nonumber}
\newcommand{\noi} {\noindent}

\newcommand{\ra} {\rightarrow}

\newcommand{\wra} {\rightharpoonup}
\newcommand{\wrastar} {\overset{\ast}{\rightharpoonup}}

\DeclareMathAlphabet{\mathpzc}{T1}{pzc}{m}{it}

\def\B{{\widetilde B}}
\def\w{{\widetilde w}}

\def\ps{p_{s}^{*}}
\def\dx{{\,\rm d}x}

\def\Dsp{{{\mathcal D}^{s,p}(\R^N)}}

\def\Dgsp{{{\mathcal D}^{s,p}_{G}(\R^N)}}

\def\sb2{{{\mathcal D}^{1,2}_0(B_1^c)}}

\def\w2r{{{ W}^{2,2}(\R^N)}}

\def\d2{{{\mathcal D}^{2,2}_0(\Om)}}

\def\C{{\mathcal C}}
\def\D{{\mathcal D}}

\def\H{{\mathcal{H}_{s,p}(\R^N)}}
\def\Hc{{\mathcal{H}_{s,p,0}(\R^N)}}

\def\R{{\mathbb R}}

\def\F{{\mathcal F}}
\def\({{\Big(}}
\def\){{\Big)}}

\def\ws2{{\F_{\frac{N}{2}}}}
\def\L2{{ L^{1,\;\infty}(\log L)^2}}
\def\dx{{\rm d}x}

\def\l2{\mathcal M\log L}

\def\c1Loc{{\C_{loc}^1}}
\def\B0n{\mathcal{B}^{}_{\lambda,0}(\mathbb{R}^N)}
\def\Bn{\mathcal{B}^{}_{\lambda,w}(\mathbb{R}^N)}

\def\Bbr{\mathcal{B}_{\lambda,w}(B_{r}(x))}
\def\Bgr{\mathcal{B}^{G}_{\lambda,w}(G(B_{r}(x))}
\def\Bg{\mathcal{B}^{G}_{\lambda,w}(\mathbb{R}^N)}

\def\Sn{\mathcal{S}^{}_{p}(\mathbb{R}^{N})}
\def\Sbn{\mathcal{S}^{}_{p}(B_{\frac{1}{n}}(x))}

\def\Sgn{\mathcal{S}^{G}_{p}(\mathbb{R}^{N})}
\def\Gagp{\iint_{\R^{2N}} \frac{|u(x)-u(y)|^p}{|x-y|^{N+sp}}\ \dxy}
\def\Gag2{\iint_{\R^{2N}} \frac{(u(x)-u(y))^2}{|x-y|^{N+sp}}\ \dxy}
\def\Gagnp{\iint_{\R^{2N}} \frac{|u_n(x)-u_n(y)|^p}{|x-y|^{N+sp}}\ \dxy}
\def\Gag1p{\iint_{\R^{2N}} \frac{|u_1(x)-u_1(y)|^p}{|x-y|^{N+sp}}\ \dxy}
\def\Gagn2{\iint_{\R^{2N}} \frac{(u_n(x)-u_n(y))^2}{|x-y|^{N+sp}}\ \dxy}

\title{Existence of Positive Solutions for Generalized Fractional Br\'{e}zis-Nirenberg Problem}
\author{Rohit Kumar and Abhishek Sarkar\thanks{corresponding author}}

\date{}
\newcommand{\Addresses}{{
  \bigskip
  \footnotesize
 Rohit Kumar,\\ Department of Mathematics,\\ Indian Institute of Technology Jodhpur, Rajasthan 342030, India. \\ \textit{E-mail address:} \texttt{rohit1.iitj@gmail.com, kumar.174@iitj.ac.in}

\vspace{2mm}

  \noindent Abhishek Sarkar,\\ Department of Mathematics,\\ Indian Institute of Technology Jodhpur, Rajasthan 342030, India. \\ \textit{E-mail address:} \texttt{abhisheks@iitj.ac.in}

}}
\begin{document}
 \maketitle \vspace{-1.5\baselineskip}
\begin{abstract}
 \noindent In this article, we study the fractional Br\'{e}zis-Nirenberg type problem on whole domain $\R^N$ associated with the fractional $p$-Laplace operator. To be precise, we want to study the following problem:
\begin{equation*} 
          (-\Delta)_{p}^{s}u - \lambda w |u|^{p-2}u= |u|^{p_{s}^{*}-2}u \quad \text{in} ~\mathcal{D}^{s,p}(\mathbb{R}^{N}), 
\end{equation*}
where $s\in (0,1),~p \in (1,\frac{N}{s}), ~p_{s}^{*}= \frac{Np}{N-sp}$ and the operator $(-\Delta)_{p}^{s}$ is the fractional $p$-Laplace operator. The space $\mathcal{D}^{s,p}(\mathbb{R}^{N})$ is the completion of $C_c^\infty(\R^N)$ with respect to the Gaglairdo semi-norm. In this article, we prove the existence of a positive solution to this problem by allowing the Hardy weight $w$ to change its sign. 
 \end{abstract} 
\medskip
\noindent
{\bf Mathematics Subject Classification (2010):} 35B09 $\cdot$ 35R11.\\
\noindent
{\bf Keywords:} Fractional Br\'{e}zis-Nirenberg problem; critical Sobolev exponent; concentration compactness; principle of symmetric criticality; positive solutions.
\maketitle
\section{Introduction}
\noindent In this article, we study the Br\'{e}zis-Nirenberg type problem associated with the fractional $p$-Laplace operator on the whole domain $ \mathbb{R}^{N}$ given by 
\begin{equation} \label{Main problem}
          (-\Delta)_{p}^{s}u - \lambda w |u|^{p-2}u= |u|^{p_{s}^{*}-2}u \quad \text{in} ~\mathcal{D}^{s,p}(\mathbb{R}^{N}) , \tag{P}
\end{equation}
where $s\in (0,1), p\in (1, \frac{N}{s})$ and the exponent $p_{s}^{*}= \frac{Np}{N-sp}$ is known as the fractional critical Sobolev exponent. The operator $(-\Delta)_{p}^{s}$ is the fractional $p$-Laplace operator defined on smooth functions as
\begin{align*}
    (-\Delta)_{p}^{s}u(x) = 2 \lim_{\epsilon \rightarrow 0^{+}} \int_{\mathbb{R}^{N} \backslash B_{\epsilon}(x)} \frac{|u(x) - u(y)|^{p-2}(u(x)-u(y))}{|x-y|^{N+sp}}\, \dy,\ \text{ for }x \in \mathbb{R}^{N},
\end{align*}
where $B_{\epsilon}(x)$ is an open ball of radius $\epsilon$ and centred at $x$. The space $\Dsp$ is defined as the completion of $C_{c}^{\infty}(\R^N)$ with respect to the Gagliardo semi-norm $\|\cdot\|_{s,p}$ defined as
\begin{align*}
    \|u\|_{s,p} := \left(\Gagp\right)^{\frac{1}{p}}.
\end{align*}
The space $\Dsp$ is known in the literature as homogeneous fractional Sobolev space. For $p\in (1,\frac{N}{s})$, this space is characterized as a function space given by $\D^{s,p}(\R^N)= \{u \in L^{\ps}(\R^N) :  \|u\|_{s,p}<+\infty\}$ (see \cite[Theorem 3.1]{Brasco-2021-characterisation}). Moreover, this is a reflexive Banach space. For more details on the space $\Dsp$, we refer to an interesting article by Brasco et al. \cite{Brasco-2021-characterisation}. For a given domain $\Omega \subset \mathbb{R}^N$, we denote $\D^{s,p}_0(\Omega)$ as the completion of $C_c^\infty(\Omega)$ with respect to the seminorm $\|\cdot\|_{s,p}$. Let $\Omega \subset \R^N$ be an open set such that $\partial \Omega$ is compact and locally the graph of a continuous function. Then for $s\in (0,1)$ and $p \in (1,\frac{N}{s})$, $\D^{s,p}_0(\Omega)$ is also characterized as a function space given by $\D^{s,p}_0(\Omega):= \{u \in L^{\ps}(\Omega) : u \equiv 0 \text{ in } \Omega^c, \|u\|_{s,p}<+\infty\}$ (see \cite[Theorem 2.1]{Brasco2016}).

For $w \equiv 0$, the problem \eqref{Main problem} reduces to
\begin{equation} \label{problem with w zero}
     (-\Delta)_{p}^{s}u = |u|^{p_{s}^{*}-2}u \quad \text{in} ~\mathcal{D}^{s,p}(\mathbb{R}^{N}).\tag{$P_0$}
\end{equation}
For $p=2$, Chen et al.\cite{Li2006} proved that every positive solution $u$ of \eqref{problem with w zero} is radially symmetric and radially decreasing about some point $x_0 \in \R^N$ and is given by
\begin{align*}
   u(x) = c \bigg( \frac{t}{t^2 + |x-x_0|^2} \bigg)^{\frac{N-2s}{2}}, 
\end{align*}
where $c$ and $t$ are positive constants. The authors use the moving plane method in an integral form and then classify solutions using that; in fact, problem \eqref{problem with w zero} is equivalent to the integral equation
\begin{align*}
    u(x) = \int_{\R^N} \frac{1}{|x-y|^{N-2s}} u(y)^{\frac{N+2s}{N-2s}}\, \dy.
\end{align*}
Furthermore, Hern\'{a}ndez-Santamar\'{\i}a and Salda\~{n}a in \cite{Hernandez2021critical} considered problem \eqref{problem with w zero} on the whole $\R^N$ or a bounded smooth domain of $\R^N$ and proved the existence as well as the convergence of solutions for $p=2$. Moreover, they pointed out that the solutions can be sign-changing depending on the symmetry conditions.

For general $p \in (1,\infty)$, we refer to {\cite[Proposition 3.1]{Brasco2016}, in which the authors have studied the existence of radial solution of \eqref{problem with w zero}} with a constant sign. Recently, in  \cite[Zhang et al.]{Zhang2023entire}, the authors have studied the existence of multiple non-radial sign-changing solutions of \eqref{problem with w zero} using the Minmax theory and the group equivariant technique along with the concentration-compactness argument. A lot of literature deals with the problem \eqref{Main problem} on a bounded domain $\Om \subset \R^N$ with Dirichlet boundary condition. For example, when $p=2$ and $w \equiv 1$, the existence of solutions was studied by Servadei and Valdinoci in \cite{Servadei2015}. Further, for general $p$ and $w \equiv 1$, we refer to \cite[Theorem 1.3]{Mosconi2016} where they proved the existence of solution for the cases $(a) \ N = sp^2 \text{ and } \lambda < \lambda_1(w)$; $(b)$ $N>sp^2$ and $\lambda$ is not one of the eigenvalues $\lambda_k$; $(c)$ $\frac{N^2}{N+s}>sp^2$; $(d)$ $\frac{N^3+s^3p^3}{N(N+s)}>sp^2 \text{ and } \partial\Omega \in C^{1,1}$. On the other hand,  for the unbounded domain $\Omega = \R^N$, $p=2$ and $w = \frac{1}{|x|^{2s}}$, the problem \eqref{Main problem} has a weak positive solution \cite[Theorem 1.5]{Dipierro2016} for certain ranges (depending on $N,p \text{ and } s$) of $\lambda$. In 2018, Bonder et al. \cite{Bonder} studied the problem \eqref{Main problem} with $\displaystyle{0 \leq w \in L^{1}_{\text{loc}}(\R^N), \ w \in L^{\infty}(\R^N)}$ with some $x_0 \in \R^N$ such that $w$ is continuous at $x_0$ and $w(x_0)>0$, and the embedding $\Dsp \subset L^{p}(w\, \dx; \R^N)$ is compact. They obtained a non-trivial weak solution \cite[Theorem 3.2]{Bonder} for $\lambda \in (0, \lambda_1(w))$, where $\lambda_1(w)$ is the first weighted eigenvalue of $(-\Delta)_{p}^{s}u = \lambda w |u|^{p-2}u$ in $\Dsp$, provided $sp^2 <N$. In 2021, Cui and Sun \cite{Cui2021problem} studied the problem by multiplying the critical term by some function $h$ on the right-hand side of \eqref{Main problem} and sign-changing $w$ with assumptions: $(a)$ $w= w_1 -w_2 \text{ with } w_1,w_2 \geq 0, w_1 \in  L^{\infty}(\mathbb{R}^{N}) \cap L^{\frac{N}{sp}}(\mathbb{R}^{N}), w_2 \in L^{\infty}(\mathbb{R}^{N}) $ ;
$(b)$ there exists constants $\rho>0$ and $\theta>1$ such that $h(x) = h(0) + o(|x|^{\frac{N}{p}})$ for $x \in B (0, \rho \theta)$;
$(c)$ $h(0)= ||h||_{\infty}$ and  $ h(x) > 0$ for $x \in B (0, \rho \theta)$ ;
$(d)$ $h \in L^{\infty}(\mathbb{R}^{N})$ and $ h^{+} \not\equiv 0$ and
$(e)$  $w{(x)} \geq w_{0} > 0 \text{ in } B(0, \rho \theta).$ They obtained at least one non-trivial solution for $\lambda \in (0, \lambda_1(w))$, where $\lambda_1(w)$ is defined as above. Li and He in \cite{Li2022} studied the existence of a positive solution to the following problem with critical nonlinearity using the mountain pass theorem 
\[\begin{cases}
 (-\Delta)_{p}^{s}u - V(x) |u|^{p-2}u=K(x)f(u) + P(x)|u|^{p_{s}^{*}-2}u, \quad x \in \R^N,\\
 u \in \mathcal{D}^{s,p}(\mathbb{R}^{N}),
\end{cases}\]
where $s \in (0,1), p\in(1,\frac{N}{s}), p_{s}^{*}= \frac{Np}{N-sp}$ and $V(x),K(x)$ are positive continuous functions that may vanish at infinity; $f: \R \rightarrow \R$ is a function with subcritical growth; and $P(x)$ is a non-negative bounded continuous function.\\
For the local case, i.e., $s=1$, Anoop and Das \cite{Anoop2023Brezis} recently obtained the positive solution of \eqref{Main problem} in a general domain $\Omega \subset \R^N$ for a specific range of $\lambda$ using the principle of symmetric criticality. For the non-local case, we refer to Bonder et al. \cite{Bonder}, where the authors studied the existence of nontrivial solutions for \eqref{Main problem} (no information about the sign of the solutions is given) provided that the weight function $w$ is non-negative, that is, it does not change sign. In this article, we prove the existence of a positive weak solution to \eqref{Main problem} by allowing the Hardy weight $w$ to change its sign. To the best of our knowledge, there is no article in this direction dealing with a positive solution to \eqref{Main problem}.
\begin{definition}[Weak solution]
    If a function $u \in \Dsp$ satisfies 
    \begin{align} \label{Def weak sol}
        \iint_{\R^{2N}} \frac{|u(x)-u(y)|^{p-2}(u(x)-u(y))(\varphi(x)-\varphi(y))}{|x-y|^{N+sp}}\, \dxy
        = \lambda \int_{\R^N}w|u|^{p-2}u\varphi\,\dx + \int_{\R^N}|u|^{\ps-2}u\varphi\,\dx,
    \end{align}
    for every $\varphi \in \Dsp$, then $u$ is called a weak solution to \eqref{Main problem}.
\end{definition}
If we assume that $u \in \Dsp$ is a weak solution to \eqref{Main problem} and if we put $\varphi=u$ as a test function in \eqref{Def weak sol}, we obtain the following estimate
\begin{align*}
    \|u\|_{s,p}^p = \lambda \int_{\R^N}w|u|^{p}\, \dx + \int_{\R^N}|u|^{\ps}\, \dx \geq \lambda \int_{\R^N}w|u|^{p}\, \dx.
\end{align*}
Clearly, the weight function $w$ needs to be chosen in such a way that it must satisfy the following fractional Hardy inequality
\begin{align}\label{FHE}
    \int_{\R^N}|w||u|^{p}\, \dx \leq C \|u\|_{s,p}^p,\ \forall \ u \in \Dsp,
\end{align}
for some constant $C>0$.
\begin{definition}[$(s,p)$-Hardy Potential]
A function $w \in L^1_{\mathrm{loc}}(\R^N)$ is said to be a $(s,p)$-Hardy potential if $w$ satisfies \eqref{FHE}. We denote the space of $(s,p)$-Hardy potentials by $\mathcal{H}_{s,p}(\R^N)$.
\end{definition}
We know that the homogeneous weight function $w(x)=\displaystyle {|x|^{-sp}}$, belongs to $\mathcal{H}_{s,p}(\R^N)$, see \cite{RS2008}. Further, by using the fractional Sobolev inequality \cite[Theorem 6.5]{DNPV2012}, we also have $L^{\frac{N}{sp}}(\R^N) \subset \mathcal{H}_{s,p}(\R^N)$.\\
 Next, we define a new space as: 
\begin{align*}
    \Hc := \overline{C_c^{\infty}(\R^N)} \ \text{in} \ \H.
\end{align*} 
Further, if we assume $\Lambda$ to be the optimal constant in the inequality \eqref{FHE} i.e., $\Lambda$ is the least possible constant so that \eqref{FHE} holds, then for $w \in \mathcal{H}_{s,p}(\R^N)$ we can write
 \begin{equation} \label{bestHardy1}
\displaystyle{\lambda_{1}(w):=\Lambda^{-1}=\inf \left\{ \|u\|_{{s,p}}^p : u \in \mathcal{D}^{s,p}({\R^N}) \text{ and } \int_{\R^N} |w| |u|^p \dx=1 \right\} }.
\end{equation}
For $w(x)= \frac{1}{|x|^{sp}}$, the above optimal constant is never achieved in $\R^N$ (see \cite{RS2008}). But for $w \in \Hc$, the best constant $\Lambda$ is achieved (see \cite{DKS}). Notice that $\lambda_{1}(w)$ is positive due to inequality \eqref{FHE}. Therefore, for any $\lambda \in (0,\lambda_{1}(w))$ we can define a quasi-norm on $\Dsp$ as follows:
\begin{align*}
    \displaystyle{\|u\|_{s,p,\lambda}:=  \bigg( \|u\|_{s,p}^p - \lambda \int_{\R^N}w|u|^p\,\dx \bigg)^{\frac{1}{p}}.}
\end{align*}
 It is easy to verify that $\|u\|_{s,p,\lambda}$ and $\|u\|_{s,p}$ are two equivalent norms.
Furthermore, {for $x \in \R^N$} we define
\begin{align*}
    \lambda_{1}(w,x) = \lim_{r\ra 0} \left[ \inf\left\{\|u\|_{s,p}^p : \int_{\R^N}|w||u|^p\,\dx=1 \text{ and } u \in \D^{s,p}_0(B_r(x))\right\} \right],
\end{align*}
and
\begin{align*}
    \Sigma_w = \left\{x \in \R^N: \lambda_{1}(w,x)<+\infty \right\}.
\end{align*}
For $w \in \H$ and $\lambda \in (0,\lambda_{1}(w))$, we consider the functional
\begin{align*}
    \mathcal{I}_{\lambda,w}(u) = \|u\|_{s,p}^p - \lambda \int_{\mathbb{R}^{N}} w|u|^p \,\dx,
\end{align*}
and define
\begin{align*}
    \mathcal{B}^{}_{\lambda,w}(\mathbb{R}^N) = \inf\left\{\mathcal{I}_{\lambda,w}(u) : u \in \mathcal{S}^{}_{p}(\mathbb{R}^{N}) \right\},
\end{align*}
where $\displaystyle{ \mathcal{S}^{}_{p}(\mathbb{R}^{N}) = \left\{ u \in \mathcal{D}^{s,p}(\R^N): \|u\|_{p_{s}^{*}}=1 \right\} }$. If $\Bn$ is attained at $v \in \mathcal{S}^{}_{p}(\mathbb{R}^{N})$, then the standard variational arguments ensure that $[\mathcal{I}_{\lambda,w}(v)]^{\frac{1}{\ps -p}}v$ is a non-trivial solution of \eqref{Main problem}. Further, we define the following concentration functions on $w$ as:
\begin{align} \label{Concentration function}
    \C_{\lambda,w}(x) := \lim_{r \ra 0} \Bbr , \ \  \C_{\lambda,w}(\infty) := \lim_{R \ra \infty} \mathcal{B}_{\lambda,w}(B_{R}^c), 
\end{align} where $\mathcal{B}_{\lambda,w}(\Omega):= \displaystyle 
 \inf_{\substack{u \in \mathcal{D}_{0}^{s,p}(\Omega)\\
    \|u\|_{p_s^*}=1}}\mathcal{I}_{\lambda,w} (u).$ We denote $\C_{\lambda,w}^{*}(\R^N) = \inf\limits_{x \in \R^N} \C_{\lambda,w}(x)$. Observe that
\begin{equation}\label{Eq both}
    \Bn \leq \C_{\lambda,w}^{*}(\R^N) \ \text{and} \ \Bn \leq \C_{\lambda,w}(\infty).
\end{equation}
In this article, we will see at a later stage that the existence of the weak solution to \eqref{Main problem} solely depends on the nature of the above inequalities, whether strict or not. For this purpose, let us make the following definition:
\begin{definition}
If $w \in \H$ and $\lambda \in (0, \lambda_{1}(w))$, then
    \begin{itemize}
        \item[(i)] $w$ is \textbf{sub-critical in} $\R^N$ if $\Bn < \C_{\lambda,w}^{*}(\R^N)$, and \textbf{sub-critical at infinity} if $\Bn < \C_{\lambda,w}(\infty)$,
        \item[(ii)]$w$ is \textbf{critical in} $\R^N$ if $\Bn = \C_{\lambda,w}^{*}(\R^N)$, and \textbf{critical at infinity} if $\Bn = \C_{\lambda,w}(\infty)$.
    \end{itemize}
\end{definition}
We recall an important theorem which will be very crucial for our subsequent results.
\begin{theorem}[\cite{DKS}]
 \label{aio}
  Let $w \in \Hc$. Then, the map, $ \mathcal{D}^{s,p}(\mathbb{R}^N) \ni u \mapsto \displaystyle\int_{\R^N} |w| |u|^p \dx$, is compact on $\Dsp$.
  \end{theorem}
Now, we state one of our main theorems regarding the existence of positive solution to \eqref{Main problem} when $w$ is subcritical in $\R^N$ and at infinity.
\begin{theorem} \label{T1}
Let $w \in \H$ be such that $w^- \in \Hc$ and $|\overline{\sum_w}|=0$. If the Hardy potential $w$ is subcritical in $\R^N$ and at $\infty$, then a positive solution to problem \eqref{Main problem} exists {  for all $\lambda \in (0, \lambda_{1}(w))$.} 
\end{theorem}
Notice that the potential $w \equiv 0$ is neither subcritical in $\R^N$ nor at $\infty$ (see Example \ref{example-1}-(i)), but the problem \eqref{Main problem} still has a positive solution (see \cite[Proposition 3.1]{Brasco2016}). Next we provide a sufficient assumptions on $w$ so that it becomes subcritical both in $\R^N$ and at $\infty$. 
\begin{theorem} \label{T}
Let $sp^2 \leq N$ and $w \in \Hc$ be such that ${w^{+} \neq 0}$ and $w \geq w_0>0 \text{ on the ball } B_{\rho \theta}(0)$, for some $\rho>0$ and $\theta>1$. Then $w$ is subcritical in $\R^N$ and at infinity for all $\lambda \in (0, \lambda_{1}(w))$.
\end{theorem}
Further, one natural question is what about the cases when $w$ does not satisfy the sub-critical condition either in $\R^N$ or at $\infty?$. For example, $w(x)=\frac{1}{|x|^{sp}}$ is not subcritical in $\R^N$ and at $\infty$ (see Example \ref{example-1}-(ii)). To study these cases, we consider some additional symmetry assumptions on the potential $w$.

For local case $i.e.,$ $s=1$, we refer to \cite{Anoop2023Brezis} 
where the authors studied the problem \eqref{Main problem} for a general domain $\Omega$. They provided several examples to show that the symmetry of the domain and the potential $w$ play an important role for the existence of solutions to \eqref{Main problem}.

We consider the action of a closed subgroup $G$ of $\mathcal{O}(N)$ (the group of all $N \times N$ orthogonal matrices) on $\mathbb{R}^{N}$ by $x \mapsto g\cdot x, \text{ for }g \in G \ \text{and} \ x \in \mathbb{R}^{N}$, where · represents the usual matrix multiplication. It is clear that the action of $G$ is linear on $\R^N$. For simplicity, we write $g(x)$ instead of $g\cdot x$. The $G$-orbit of any point $x \in \R^N$ is denoted by $Gx$ and defined as $Gx = \{g(x) : g \in G\}$, and the $G$-orbit of any subset $E \subset \R^N$ is denoted by $G(E)$ and is defined as $G(E) = \{g(x) : x \in E, g \in G\}$.
\begin{definition}
    Let $\Omega$ be a domain in $\mathbb{R}^{N}$ and $f: \Omega \rightarrow \mathbb{R}$ be any given function. If $G(\Omega) = \Omega$, then we say $\Omega$ is a \textbf{$G$-invariant} domain. If $\Omega$ is a $G$-invariant domain and $f(g(x)) = f(x),~\forall~g\in G$ (i.e., $f$ is constant on each $G$-orbit), then we say $f$ is a $G$-invariant function.
\end{definition}
It is easy to observe that $\Omega=\mathbb{R}^{N}$ is $G$-invariant due to an identity element in $G$ and the action of $G$ on $\mathbb{R}^{N}$ naturally induces an action of $G$ on $\D^{s,p}(\mathbb{R}^{N})$ given by
\begin{align*}
    \pi(g)(u) = u_g,~\text{where}~u_g(x)=u(g^{-1}x),~\forall~x \in \mathbb{R}^{N}.
\end{align*}
Thus, $u \in \Dsp$ is $G$-invariant if and only if $u_g=u$, for each $g \in G$. The set of all $G$-invariant functions in $\Dsp$ and $\mathcal{S}^{}_{p}(\mathbb{R}^{N})$ are denoted by $\Dgsp$ and $\mathcal{S}^{G}_{p}(\mathbb{R}^{N})$ respectively, i.e.,
\begin{align*}
    \mathcal{D}_G^{s,p}(\R^N) &= \left\{ u \in \mathcal{D}^{s,p}(\R^N): u_g=u, \forall g \in G \right\},\\
    \mathcal{S}^{G}_{p}(\mathbb{R}^{N}) &= \left\{ u \in \mathcal{S}^{}_{p}(\mathbb{R}^{N}): u_g=u, \forall g \in G  \right\}.
\end{align*}
Now we consider $G$-dependent minimization problem analogous to $\mathcal{B}^{}_{\lambda,w}(\mathbb{R}^N)$ as
\begin{align*}
  \mathcal{B}^{G}_{\lambda,w}(\mathbb{R}^N) = \inf\left\{\mathcal{I}_{\lambda,w}(u) : u \in \mathcal{S}^{G}_{p}(\mathbb{R}^{N}) \right\}.  
\end{align*}
It is clear that $ \mathcal{B}^{}_{\lambda,w}(\mathbb{R}^N) \leq \mathcal{B}^{G}_{\lambda,w}(\mathbb{R}^N)$. It might be the case that $\mathcal{B}^{G}_{\lambda,w}(\mathbb{R}^N)$ is attained
in $\mathcal{D}_G^{s,p}(\R^N)$, but meanwhile $\mathcal{B}^{}_{\lambda,w}(\mathbb{R}^N)$ is not attained in $\mathcal{D}^{s,p}(\R^N)$. So a natural question arises whether a minimizer of $\mathcal{B}^{G}_{\lambda,w}(\mathbb{R}^N)$ actually solves (\ref{Main problem}) or not? i.e., whether a critical point of $\mathcal{I}_{\lambda,w}$ over $\mathcal{S}^{G}_{p}(\mathbb{R}^{N})$ can be a critical point of $\mathcal{I}_{\lambda,w}$ over $\mathcal{S}^{}_{p}(\mathbb{R}^{N})$? Here the role of the principle of symmetric criticality theory comes into the picture, which answers this question affirmatively. The version of the principle of symmetric criticality which we used is given by Kobayashi and \^Otani \cite[Theorem 2.2]{Kobayashi2004} (see Theorem \ref{PSC}). Further, we define the following $G$-dependent concentration function on $w$ as:
\begin{align} \label{G Concentration function}
    \C^{G}_{\lambda,w}(x) := \lim_{r \ra 0} \Bgr , \ \  \C^{G}_{\lambda,w}(\infty) := \lim_{R \ra \infty} \mathcal{B}^{G}_{\lambda,w}(B_R^c). 
\end{align}
We denote {$\mathcal{B}^G_{\lambda,w}(\Omega):= \displaystyle 
 \inf_{\substack{u \in \mathcal{D}_{0,G}^{s,p}(\Omega)\\
    \|u\|_{p_s^*}=1}}\mathcal{I}_{\lambda,w} (u)$} and $\C_{\lambda,w}^{G,*}(\R^N) = \inf\limits_{x \in \R^N} \C^{G}_{\lambda,w}(x)$.\\
    The space $\mathcal{D}_{0,G}^{s,p}(\Omega)$ is given by
    \begin{align*}
        \mathcal{D}_{0,G}^{s,p}(\Omega) = \left\{ u \in \mathcal{D}_0^{s,p}(\Omega): u_g=u, \forall g \in G \right\}.
    \end{align*}
We notice that
\begin{equation}\label{G Eq both}
    \Bg \leq \C_{\lambda,w}^{G,*}(\R^N) \ \text{and} \ \Bg \leq \C^{G}_{\lambda,w}(\infty).
\end{equation}
Similarly, we will see that the existence of a solution to \eqref{Main problem} also depends solely on the nature of the above inequalities, whether strict or not. For this purpose, let us make the following definition analogously:
\begin{definition}
    If $w \in \H$ is $G$-invariant and $\lambda \in (0, \lambda_{1}(w))$, then
    \begin{itemize}
        \item[(i)] $w$ is \textbf{$G$-subcritical in} $\R^N$ if $\Bg < \C_{\lambda,w}^{G,*}(\R^N)$, and \textbf{$G$-subcritical at infinity}, at level $\lambda$ if $\Bg < \C^G_{\lambda,w}(\infty)$.
        \item[(ii)]$w$ is \textbf{$G$-critical in} $\R^N$ if $\Bg = \C_{\lambda,w}^{G,*}(\R^N)$, and \textbf{$G$-critical at infinity}, at level $\lambda$ if $\Bg = \C^{G}_{\lambda,w}(\infty)$.
    \end{itemize}
\end{definition} 
Next, we state another important sets of results regarding the existence of solution to \eqref{Main problem}.
\begin{theorem} \label{T2}
Let $w \in \H$ be a $G$-invariant Hardy potential such that $w^- \in \Hc$ and $|\overline{\sum_w}|=0$. If the Hardy potential $w$ is $G$- subcritical in $\R^N$ and at $\infty$, then there is a positive solution to \eqref{Main problem} {  for all $\lambda \in (0, \lambda_{1}(w))$.} 
\end{theorem}
{  In the next result, we consider the case when $w$ need not be $G$-subcritical in $\R^N$.}
\begin{theorem} \label{T3}
Let $w \in \H$ be a $G$-invariant Hardy potential such that $w^- \in \Hc$ and $|\overline{\sum_w}|=0$. If the orbits $Gx$ are infinite for all $x \in \R^N$ and the Hardy potential $w$ is $G$- subcritical at $\infty$ for some $\lambda \in (0, \lambda_1(w))$. Then there is a positive solution to the problem \eqref{Main problem}.
\end{theorem}
Furthermore, we discuss the case when $w$ need not be $G$-subcritical at infinity. Here, we assume the following condition on $w$:
\begin{align} \label{condition on w critical case}
   w(rz) \geq \frac{w(z)}{r^{sp}}; \ z \in \R^N,\ r>0.
\end{align} 
Observe that $w\equiv 0$ or $w(x)=\frac{1}{|x|^{sp}}$ satisfies \eqref{condition on w critical case}. Following the proof of Example \ref{example-1}, it can be verified that these potentials are not $G$-subcritical at infinity. The following theorem ensures that the positive solutions to \eqref{Main problem} exist for such potentials.
\begin{theorem} \label{critical case infinity}
    Let $G$ be a closed subgroup of $\mathcal{O}(N)$. Let $w \in \H$ be a $G$-invariant Hardy potential such that $w^- \in \Hc$ and $|\overline{\sum_w}|=0$. If one of the following holds, i.e.
    \begin{itemize}
        \item[(a)] $w$ is $G$-subcritical in $\R^N$ for some $\lambda \in (0, \lambda_1(w))$ and satisfies \eqref{condition on w critical case} for sufficiently small $r>0$. 
        \item[(b)] The orbits $Gx$ are infinite for all $x \in \R^N \setminus \{0\}$ and satisfy \eqref{condition on w critical case} for all values of $r>0$.
    \end{itemize}
 Then a positive solution to the problem \eqref{Main problem} exists for $\lambda \in (0, \lambda_1(w))$. 
\end{theorem} 
The rest of the article is organized as follows. In Section \ref{Prelim}, we have discussed some preliminary results and useful tools to proceed with our exploration. In Section \ref{MR}, we have discussed the proofs of the main results. 

\section{Preliminaries}\label{Prelim}
In this section, we briefly recall some preliminary results to be used in the subsequent sections.
\subsection{The space of signed measures} 
 We denote the collection of all Borel sets in $\R^N$ by $\mathbb{B}(\R^N)$. Let $\mathbb{M}(\R^N)$ be the space of all regular, finite,  Borel-signed measures on $\R^N.$
 Then $\mathbb{M} (\R^N)$ is a Banach space to the norm $\|\nu\|=|\nu|(\R^N)$ (total variation of the measure $\nu$). By Riesz representation theorem, we know that $\mathbb{M}(\R^N)$ is the dual of $\text{C}_0(\R^N)$ := $\overline{\text{C}_c(\R^N)}$ in $L^{\infty}(\R^N)$ \cite[Theorem 14.14, Chapter 14]{Border}. The next proposition follows from the uniqueness part of the Riesz representation theorem.
\begin{proposition} \label{defmeasure}
 Let $\nu \in \mathbb{M}(\R^N)$ be a positive measure. Then for an open set $V \subseteq \R^N$,
 \begin{align*}
     \nu(V)= \sup \left \{ \int_{\R^N} \phi \ \mathrm{d}\nu : 0 \leq \phi \leq 1, \phi \in \rm{C}_c^{\infty}(\R^N) \ with \ Supp(\phi) \subseteq V   \right \}
 \end{align*}
and for any Borel set $E \subseteq \R^N$, $\nu(E):= \inf \left\{ \nu(V) : E \subseteq V \ \mbox{and} \ V \text {is open} \right\}$.
\end{proposition}
A  sequence $(\nu_n)$ is said to be weak$^*$ convergent to $\nu$ in $\mathbb{M}(\R^N)$, if
 \begin{align*}
  \int_{\R^N} \phi \ \mathrm{d}\nu_n \ra  \int_{\R^N} \phi \ \mathrm{d}\nu, \ as \ n \ra \infty,  \forall\,  \phi \in \text{C}_0(\R^N).
 \end{align*}
 In this case, we denote $\nu_n \wrastar \nu$. The following Proposition is a consequence of the Banach-Alaoglu theorem \cite[Chapter 5, Section 3]{Conway}, which states that for any normed linear space $X$, the closed  unit ball in  $X^*$ is weak$^*$ compact.
\begin{proposition} \label{BanachAlaoglu}
 Let $(\nu_n)$ be a bounded sequence in $\mathbb{M}(\R^N)$, then there exists $\nu \in \mathbb{M}(\R^N)$ such that $\nu_n \overset{\ast}{\rightharpoonup} \nu$ up to a subsequence.
\end{proposition}
\begin{proof}
Recall that, if $X=\rm{C}_0(\R^N)$ then by Riesz Representation theorem \cite[Theorem 14.14, Chapter 14]{Border} $X^*=\mathbb{M}(\R^N)$. Thus, the proof follows from the Banach-Alaoglu theorem \cite[Chapter 5, Section 3]{Conway}. 
\end{proof} 
 We define the fractional $(s,p)$-gradient of a function $u \in \Dsp$ as
\begin{align*}
    |D^su(x)|^p:= \int_{\R^N}\frac{|u(x+h)-u(x)|^p}{|h|^{N+sp}}\, \mathrm{d} h.
\end{align*} 
Observe that this $(s,p)$-gradient is well-defined a.e. in $\R^N$ and moreover $|D^su| \in L^p(\R^N)$.
 Now, let $u_n,u \in \mathcal{D}^{s,p}(\R^N)$. For a Borel set  $ E $ in   $\R^N,$ we  denote 
\begin{align*}
\nu_n(E)&:=\int_E |u_n - u|^{p_{s}^{*}} \ \dx, \quad 
\Ga_n(E):=\int_E|D^s (u_n-u)|^p \ \dx,  \\ 
\widetilde{\Ga}_n(E) &:=\int_E|D^s u_n|^p \ \dx \, ,\quad
\gamma_{n}(E):= \int_E w|u_n - u|^{p} \ \dx, \  \\
\ \widetilde{\gamma}_n(E)&:= \int_E w|u_n|^{p} \ \dx \,.
 \end{align*}
For a sequence $(u_n)$ in $\mathcal{D}^{s,p}(\R^N)$ with $u_n\wra u$ in $\mathcal{D}^{s,p}(\R^N)$, each of the sequences $\nu_n$, $\Ga_n, \widetilde{\Ga}_n , \gamma_n$ and $\widetilde{\gamma}_n$ correspond to a bounded sequence in $\mathbb{M}(\R^N)$ and 
  have weak$^*$ convergent sub-sequences (Proposition \ref{BanachAlaoglu}) in $\mathbb{M}(\R^N)$. Let
  \begin{equation} \label{weak star convergence of measures}
      \nu_n \overset{\ast}{\rightharpoonup} \nu \,, \quad \Gamma_n \overset{\ast}{\rightharpoonup} \Gamma \,, \quad \widetilde{\Gamma}_n \overset{\ast}{\rightharpoonup} \widetilde{\Gamma} \,,\quad {\gamma}_n \overset{\ast}{\rightharpoonup} {\gamma} \, \text{  and  } \widetilde{\gamma}_n \overset{\ast}{\rightharpoonup} \widetilde{\gamma} \  \mbox{ in } \mathbb{M}(\R^N).
  \end{equation}
Further, we define the following limits:
\begin{align*}
  \lim\limits_{R \rightarrow \infty} \overline{ \lim\limits_{n \rightarrow \infty}}  \int_{|x|\geq R} |u_n - u|^{p_{s}^{*}} \ \dx \, &= \nu_\infty,\\
    \lim\limits_{R \rightarrow \infty} \overline{ \lim\limits_{n \rightarrow \infty}}\int_{|x|\geq R} |D^s (u_n-u)|^p \ \dx \, &= \Ga_\infty,\\
    \lim\limits_{R \rightarrow \infty} \overline{ \lim\limits_{n \rightarrow \infty}}\int_{|x|\geq R} w|u_n - u|^{p} \ \dx \, &= \gamma_\infty.
\end{align*}
 Next, we recall the following result by Lions \cite[Lemma 1.2]{Lionslimitcase}, which we use further to prove the concentration-compactness results.
\begin{lemma} \label{Lions lemma}
   Let $\nu,\Gamma$ be two non-negative, bounded measures on $\mathbb{R}^{N}$ such that
   \[ \bigg[ \int_{\mathbb{R}^N} |\phi|^q \mathrm{d}\nu \bigg]^{\frac{1}{q}} \leq C \bigg[ \int_{\mathbb{R}^N} |\phi|^p \mathrm{d}\Gamma \bigg]^{\frac{1}{p}},~~ \forall ~\phi \in C_{c}^{\infty}(\mathbb{R}^N), \]
   for some constant $C>0$ and $1\leq p<q<\infty$. Then there exist a countable set $\left\{x_k \in \mathbb{R}^N : k \in \mathbb{K}\right\}$ and $\nu_k \in (0, \infty)$ such that 
   $ \nu = \displaystyle\sum_{k \in \mathbb{K}} \nu_k \delta_{x_k}.$
\end{lemma}
\begin{definition}
A measure $\Upsilon \in \mathbb{M}(\R^N)$ is said to be concentrated on a Borel set $F$ if
\begin{align*}
    \Upsilon(E) = \Upsilon(E \cap F), \ \forall E \in \mathbb{B}(\R^N).
\end{align*}
If $\Upsilon$ is concentrated on $F$, then one can observe that $\Upsilon(F) = \|\Upsilon \|$.
For any measure $\Upsilon \in \mathbb{M}(\R^N)$ and a $E \in \mathbb{B}(\R^N)$, we denote the restriction of
$\Upsilon$ on $E$ as $\Upsilon_E$. Observe that $\Upsilon_E$ is concentrated on $E$.
\end{definition}
\subsection{Strong maximum principle}
We use the following strong maximum principle given by Quass and Pezzo \cite[Theorem 1.2]{Quass2017} to show the positivity of the solution for \eqref{Main problem}.
\begin{lemma}
Let $w \in L^{1}_{\text{loc}}(\mathbb{R}^{N})$ be a non-negative function and $u \in \mathcal{D}^{s,p}(\mathbb{R}^{N})$ be a weak super-solution of $(-\Delta)_{p}^{s}u + \lambda w |u|^{p-2}u= 0~\text{in}~\mathbb{R}^{N}.$ If $u\geq 0$ a.e. in $\mathbb{R}^{N}$, then either $u>0$ a.e. in $\mathbb{R}^{N}$ or $u=0$ a.e in $\mathbb{R}^{N}$.
\end{lemma}
\subsection{Br\'{e}zis-Lieb Lemma}
 The next lemma is due to Br\'{e}zis and Lieb  \cite[Theorem 1]{Brezis_Lieb_1983}.
\begin{lemma} \label{BrezisLieb}
Let $(\Omega, \mathcal{A},\mu)$ be a measure space and $\langle f_{n} \rangle$ be a sequence of complex-valued measurable functions which are uniformly bounded in $L^{p}(\Omega,\mu)$ for some $0<p<\infty$. Moreover, if $\langle f_{n} \rangle$ converges to $f$ a.e., then
\begin{align*}
    \lim\limits_{n \ra \infty} \left| \|f_{n}\|_{p} - \|f_{n}-f\|_{p} \right| = \|f\|_{p}.
\end{align*}
\end{lemma}
\subsection{Principle of symmetric criticality}
Let $\left(X, \| \cdot \|_X \right)$ be a real Banach space. Let $G$ be a group and $\pi$ be a representation of $G$ over $X$, i.e., $\pi(g)$ is a bounded linear operator in $X$ for each $g \in G$ and satisfies:
\begin{itemize}
    \item[$(i)$] $\pi(e)u = u, \  \forall u \in X$ , where $e$ is the identity element in $X$.
    \item [$(ii)$] $\pi(g_1 g_2)u = \pi(g_1)(\pi(g_2)u), \forall g_1, g_2 \in G, u \in X.$
\end{itemize}
The representation $\pi_*$ of $G$ over $X^*$ is induced by $\pi$ and it is given by
\begin{align*}
    \langle \pi_*(g)v^*,u\rangle = \langle v^*,\pi(g^{-1})u \rangle,~g\in G,~u \in X,~v^* \in X^*, 
\end{align*}
where $\langle \cdot,\cdot \rangle$ denotes the duality product between $X$ and $X^*$.
The action of $G$ on $X$ is said to be isometric if $\|\pi(g)u\|_X = \|u\|_X , \forall g \in G,
u \in X$. The subspace $ \Sigma := \{ u \in X : \pi(g)u = u, \forall g \in G\}$ is called the
$G$-invariant subspace of $X$. A functional $J : X \rightarrow \mathbb{R}$ is called $G$-invariant if
$J(\pi(g)u) = J(u), \forall g \in G, \forall u \in X$. A continuously differentiable $G$-invariant functional
$J : X \rightarrow \mathbb{R}$ is said to satisfy the principle of symmetric criticality if 
\begin{align*}
    \left(J|_\Sigma \right)'(u) =0 \ \text{implies} \ J'(u) =0 \ \text{and} \ u \in \Sigma.
\end{align*}
In 1979, Palais \cite{Palais1979} introduced the notion of the principle of symmetric criticality. Since then, many versions of this principle have been proved e.g., \cite[Theorem
2.2]{Kobayashi2004}, \cite[Theorem 2.7]{Kobayashi2004}, \cite[Theorem 2.1]{Alexandru2007}. In this article, we use the following
version due to Kobayashi and \^{O}tani \cite[Theorem
2.2]{Kobayashi2004}.
\begin{theorem}[Principle of symmetric criticality]\label{PSC}
    Let $X$ be a reflexive and strictly convex Banach space and $G$ be a group that acts on $X$ isometrically i.e., $\|gu\|_X = \|u\|_X, \ \forall g \in G, \forall u \in X $. If $J$ is a $G$-invariant $C^1$-functional on $X$, then
    \begin{align*}
        \left(J|_{\Sigma}\right)'(u)=0 \text{ implies } J'(u)=0 \text{ and } u \in \Sigma,
    \end{align*}
    where $\Sigma$ is the set of all $G$-invariant elements of $X$. Here, $\left(J|_{\Sigma}\right)'(u)$ and $ J'(u)$ denote the Fr\'{e}chet derivatives of $J|_{\Sigma}$ and $J$ at $u$ in $\Sigma$ and $X$, respectively.
\end{theorem}
Let us recall that the space $\D^{s,p}(\mathbb{R}^{N})$ is the completion of $C_c^{\infty}(\mathbb{R}^{N})$ with respect to the Gagliardo semi-norm given by
\begin{align*}
    \|u\|_{s,p} = \bigg(\iint_{\mathbb{R}^{2N}} \frac{|u(x)-u(y)|^p}{|x-y|^{N+sp}}\, \dxy \bigg)^{\frac{1}{p}}.
\end{align*}
For $p>1$, the space $\D^{s,p}(\mathbb{R}^{N})$ is reflexive and strictly convex Banach space. Next, we show that the subgroup $G$ acts isometrically on $\D^{s,p}(\mathbb{R}^{N})$.
\begin{align*}
    &\|g(u)\|_{s,p}^p = \|\pi(g)(u)\|_{s,p}^p=\|u_g\|_{s,p}^p= \iint_{\mathbb{R}^{N} \times \mathbb{R}^{N}} \frac{|u_g(x)-u_g(y)|^p}{|x-y|^{N+sp}}\, \dxy \\
    &= \frac{1}{|\text{det}(g^{-1})|^2} \iint\limits_{G(\mathbb{R}^{N}) \times G(\mathbb{R}^{N})} \frac{|u(x)-u(y)|^p}{|g(x)-g(y)|^{N+sp}}\, \dxy,~(\text{by \cite[Theorem 2.44 (a)]{Folland}})\\
    &= \iint\limits_{\mathbb{R}^{N} \times \mathbb{R}^{N}} \frac{|u(x)-u(y)|^p}{|g(x)-g(y)|^{N+sp}}\, \dxy,\text{ (since } \Omega=\R^N \text{ is G-invariant)}.
\end{align*}
Since the action of the subgroup $G$ on $\mathbb{R}^{N}$ is linear, using the linearity of $g$ we have
\begin{align*}
    |g(x)-g(y)|^2 = \langle g(x-y), g(x-y) \rangle = \langle x-y, g^T g(x-y) \rangle = \langle x-y, x-y \rangle = |x-y|^2.
\end{align*}
Using the above equality, we obtain
\begin{align*}
    \iint_{\mathbb{R}^{2N}} \frac{|u(x)-u(y)|^p}{|g(x)-g(y)|^{N+sp}}\, \dxy = \iint_{\mathbb{R}^{2N}} \frac{|u(x)-u(y)|^p}{|x-y|^{N+sp}}\, \dxy.
\end{align*}
Thus $ \|g(u)\|_{s,p} = \|u\|_{s,p}$ i.e., the action of $G$ on $\D^{s,p}(\mathbb{R}^{N})$ is isometric.
\section{Proofs of the main results}\label{MR}
In this section, we demonstrate the proofs of our main results. Before that, we state and prove the next proposition.
\begin{proposition} \label{nu infinity prop}
   If for each $R>0$, we consider a function $\Phi_R \in C_b^1(\R^N)$ (i.e., continuously differentiable and bounded) satisfying $0\leq \Phi_R \leq 1,~\Phi_R = 0$ on $\overline{B_R}$ and $\Phi_R = 1$ on ${B_{R+1}^c}$. Then, for $u_n \rightharpoonup u$ weakly in $\Dsp$, the following statements hold:
   \begin{itemize}
       \item[(i)] $\lim\limits_{R \rightarrow \infty} \overline{ \lim\limits_{n \rightarrow \infty}}  \int_{\R^N}|u_n |^{p_{s}^{*}} \Phi_R \ \dx \, = \nu_\infty$.
       \item[(ii)] $\lim\limits_{R \rightarrow \infty} \overline{ \lim\limits_{n \rightarrow \infty}}\int_{\R^N} w|u_n|^{p} \Phi_R \ \dx \, = \gamma_\infty.$
       \item[(iii)] $\lim\limits_{R \rightarrow \infty} \overline{ \lim\limits_{n \rightarrow \infty}}\int_{\R^N} |D^s u_n|^p  \Phi_R\ \dx \, = \Ga_\infty.$
   \end{itemize}
\end{proposition}
 \begin{proof}
Since $u_n \rightharpoonup u$ weakly in $\mathcal{D}^{s,p}(\mathbb{R}^{N})$. The Br\'{e}zis-Lieb Lemma \ref{BrezisLieb} yields
\begin{equation} \label{one}
    \overline{\lim\limits_{n \rightarrow \infty} }\bigg|\int_{|x|\geq R} |u_n -u|^{p_{s}^{*}}\ \dx - \int_{|x|\geq R} |u_n |^{p_{s}^{*}}\ \dx\bigg| = \int_{|x|\geq R} |u|^{p_{s}^{*}}\ \dx.
\end{equation}
Using the fact that limsup is sub-additive, we have
\begin{align} \label{two}
\begin{split}
     \left| \overline{\lim\limits_{n \rightarrow \infty}}\int_{|x|\geq R} |u_n -u|^{p_{s}^{*}}\ \dx - \overline{\lim\limits_{n \rightarrow \infty}}\int_{|x|\geq R} |u_n |^{p_{s}^{*}}\ \dx\right|
    \leq \overline{\lim\limits_{n \rightarrow \infty}} \left|\int_{|x|\geq R} |u_n -u|^{p_{s}^{*}}\ \dx - \int_{|x|\geq R} |u_n |^{p_{s}^{*}}\ \dx \right|.
\end{split}
\end{align}
From \eqref{one} and \eqref{two} we deduce that
\begin{align*}
    \left| \overline{\lim\limits_{n \rightarrow \infty}}\int_{|x|\geq R} |u_n -u|^{p_{s}^{*}}\ \dx - \overline{\lim\limits_{n \rightarrow \infty}}\int_{|x|\geq R} |u_n |^{p_{s}^{*}}\ \dx\right| \leq \int_{|x|\geq R} |u|^{p_{s}^{*}}\ \dx.
\end{align*}
Now taking the limit as $R \rightarrow \infty$ yields
\begin{align*}
    \left| \lim\limits_{R \rightarrow \infty}\overline{\lim\limits_{n \rightarrow \infty}}\int_{|x|\geq R} |u_n -u|^{p_{s}^{*}}\ \dx -\lim\limits_{R \rightarrow \infty} \overline{\lim\limits_{n \rightarrow \infty}}\int_{|x|\geq R} |u_n |^{p_{s}^{*}}\ \dx\right| =0.
\end{align*}
Thus we obtain,
\begin{equation} \label{three}
    \lim\limits_{R \rightarrow \infty} \overline{\lim\limits_{n \rightarrow \infty}}\int_{|x|\geq R} |u_n |^{p_{s}^{*}}\ \dx = \lim\limits_{R \rightarrow \infty}\overline{\lim\limits_{n \rightarrow \infty}}\int_{|x|\geq R} |u_n -u|^{p_{s}^{*}}\ \dx= \nu_\infty.
\end{equation}
Moreover, we have the following estimate
\begin{align*}
    \int_{|x|\geq R+1} |u_n |^{p_{s}^{*}}\ \dx =\int_{|x|\geq R+1} |u_n |^{p_{s}^{*}} \Phi_R\ \dx \leq \int_{\R^N} |u_n |^{p_{s}^{*}} \Phi_R\ \dx
    = \int_{|x|\geq R} |u_n |^{p_{s}^{*}} \Phi_R\ \dx \leq \int_{|x|\geq R} |u_n |^{p_{s}^{*}} \ \dx.
\end{align*}
From the above inequalities, we get
\begin{align*}
    \int_{|x|\geq R+1} |u_n |^{p_{s}^{*}}\ \dx \leq \int_{\R^N} |u_n |^{p_{s}^{*}} \Phi_R\ \dx \leq \int_{|x|\geq R} |u_n |^{p_{s}^{*}} \ \dx.
\end{align*}
By taking $n,R \rightarrow \infty$ and using \eqref{three}, we obtain
\begin{align*}
    \lim\limits_{R \rightarrow \infty} \overline{\lim\limits_{n \rightarrow \infty}}\int_{\R^N} |u_n |^{p_{s}^{*}} \Phi_R\ \dx = \lim\limits_{R \rightarrow \infty} \overline{\lim\limits_{n \rightarrow \infty}}\int_{|x|\geq R} |u_n |^{p_{s}^{*}}\ \dx = \nu_\infty.
\end{align*}
Hence, the proof of part $(i)$ is complete. Part $(ii)$ and $(iii)$ can be obtained by following similar arguments.
\end{proof}
The following lemma is due to \cite[Lemma 2.4]{Bonder} and \cite[Remark 2.5]{Bonder}.
\begin{lemma} \label{lem: 1}
Let $s\in (0,1),p\in (1,\frac{N}{s})$ and $\phi \in W^{1,\infty}(\R^N)$ with compact support.
Let $u_n \wra u$ in $\mathcal{D}^{s,p}(\R^N)$. Then
$$\displaystyle \lim_{n\ra \infty} \int_{\R^N} |u_n(x)-u(x)|^p |D^s \phi(x)|^p \dx = 0 \,.$$
\end{lemma}
\begin{remark} \label{infestimate} \rm 
Notice that for $\phi \in W^{1,\infty}(\R^N)$ with compact support, by \cite[Lemma 2.2]{Bonder}, we have $|D^s \phi| \in L^{\infty}(\R^N).$ Moreover,
\begin{align*}
    |D^s \phi(x)|^p \leq \text{C} \min\{1,|x|^{-(N+sp)}\} \,,
\end{align*}
where $\text{C}>0$ depends on $N,s,p$ and $\|\phi\|_{W^{1,\infty}}$.
Consequently, $|D^s \phi| \in L^p(\R^N).$
Now, let $\psi \in C_b^{\infty}(\R^N)$ be such that $0\leq \psi\leq 1$, $\psi=0$ on $B_1(0)$, and $\psi =1$ on $B_2(0)^c$. Then, $\phi := 1-\psi \in W^{1,\infty}(\R^N)$ with support in $B_2(0)$ and $|D^s \psi|^p = |D^s \phi|^p$. Thus,
\begin{align*}
    |D^s \psi(x)|^p \leq \text{C} \min\{1,|x|^{-(N+sp)}\} \,,
\end{align*}
where $\text{C}>0$ depends on $N,s,p$ and $\|\psi\|_{W^{1,\infty}}$. Now the weight function $w= |D^s \psi|^p$ also verifies the hypothesis of \cite[Lemma 2.4]{Bonder} with $q=p$. Hence, we conclude that Lemma \ref{lem: 1} also holds if we replace $\phi$ with  $\psi$.
\end{remark} 
\begin{corollary} \label{passing_limit}
Let $u_n \wra u$ in $\mathcal{D}^{s,p}(\R^N)$. Let $\phi \in W^{1,\infty}(\R^N)$ with compact support or $\phi \in C_b^{\infty}(\R^N)$ with $0 \leq \phi \leq 1$, $\phi = 0$ on $B_1(0)$ and $\phi =1$ on $B_2(0)^c$. Then, for $v_n=(u_n-u)\phi$, we have
\begin{align*}
 \lim_{n\ra \infty} &\iint_{\R^{2N}} \frac{|v_n(x)-v_n(y)|^p}{|x-y|^{N+sp}} \dxy \,  \leq  \ \lim_{n\ra \infty} \iint_{\R^{2N} } |\phi(y)|^p\frac{|(u_n-u)(x)-(u_n-u)(y)|^p}{|x-y|^{N+sp}} \dxy . 
\end{align*}
\end{corollary}
\begin{proof}
Let $w_n=u_n-u$. Then $v_n=w_n \phi$. Applying Minkowski's inequality and subsequently changing variables gives the following inequalities
\begin{align*}
  \|v_n\|_{s,p} & \leq \left(  \iint_{\R^{2N}} \frac{|w_n(x)-w_n(y)|^p}{|x-y|^{N+sp}} |\phi(y)|^p\,\dxy \right)^{\frac{1}{p}} +\left(  \iint_{\R^{2N}} \frac{|\phi(x)-\phi(y)|^p}{|x-y|^{N+sp}} |w_n(x)|^p\,\dxy \right)^{\frac{1}{p}} \\
  &\leq \left(  \iint_{\R^{2N}} \frac{|w_n(x)-w_n(y)|^p}{|x-y|^{N+sp}} |\phi(y)|^p\,\dxy \right)^{\frac{1}{p}} +\left(  \iint_{\R^{2N}} |w_n(x)|^p |D^s \phi(x)|^p\,\dx \right)^{\frac{1}{p}}.
\end{align*}
Notice that the second integral in the above inequality tends to $0$ as $n \rightarrow \infty$ using Lemma \ref{lem: 1} and Remark \ref{infestimate}. Hence, by taking limit as $n \rightarrow \infty$ in the above inequality, we get
\begin{align*}
     \lim_{n\ra \infty} \iint_{\R^{2N}} \frac{|v_n(x)-v_n(y)|^p}{|x-y|^{N+sp}} \dxy \leq \lim_{n\ra \infty}  \iint_{\R^{2N}} \frac{|w_n(x)-w_n(y)|^p}{|x-y|^{N+sp}} |\phi(y)|^p\,\dxy,
\end{align*}
which gives the desired result. 
\end{proof}
\begin{proposition}
Let G be a closed subgroup of $\mathcal{O}(N)$, $w \in \H$ be a non-negative $G$-invariant Hardy potential, and $(u_n)$ be a sequence in $\mathcal{D}_G^{s,p}(\R^N)$ such that $u_n \rightharpoonup u$ in $\mathcal{D}_G^{s,p}(\R^N)$. Then the following statements are true:
\begin{itemize}
    \item[(i)] there exists a countable set $\mathbb{K}$ such that $ \nu = \sum\limits_{k \in \mathbb{K}} \nu_k \delta_{x_k}$, where $\nu_k \in (0, \infty), x_k \in \R^N$. In particular, $\nu$ is supported on the countable set $A_{\mathbb{K}}:= \{x_k \in \R^N : k \in \mathbb{K}\}$,
    \item[(ii)] $\gamma$ is supported on $\overline{\Sigma_w}$.
\end{itemize}
\end{proposition}
\begin{proof}
$(i)$ For $\lambda \in (0, \lambda_1(w))$, the definition of $\mathcal{B}^{G}_{\lambda,w}(\mathbb{R}^N)$ yields
\begin{align} \label{p1eq1}
    \mathcal{B}^{G}_{\lambda,w}(\mathbb{R}^N) &\leq \frac{\displaystyle\iint_{\mathbb{R}^{2N}} \frac{|v_n(x)-v_n(y)|^p}{|x-y|^{N+sp}} \,\dxy - \lambda \displaystyle\int_{\mathbb{R}^{N}} w|v_n|^p \,\dx}{ \bigg[ \displaystyle\int_{\mathbb{R}^{N}} |v_n|^{p_{s}^{*}} \,\dx \bigg]^{\frac{p}{p_{s}^{*}}}} \leq \frac{\displaystyle\iint_{\mathbb{R}^{2N}} \frac{|v_n(x)-v_n(y)|^p}{|x-y|^{N+sp}} \,\dxy}{ \bigg[ \displaystyle\int_{\mathbb{R}^{N}} |v_n|^{p_{s}^{*}} \,\dx \bigg]^{\frac{p}{p_{s}^{*}}}},
\end{align}
where $v_n$ is defined as $v_n = (u_n -u) \phi$, for any $G$-invariant $\phi \in C_{c}^{\infty}(\mathbb{R}^N)$. Taking limit as $n \rightarrow \infty$ in \eqref{p1eq1} and using Corollary \ref{passing_limit}, we obtain
\begin{align*}
    \mathcal{B}^{G}_{\lambda,w}(\mathbb{R}^N) &\lim\limits_{n \rightarrow \infty}\bigg[ \int_{\mathbb{R}^{N}} |(u_n -u) \phi|^{p_{s}^{*}} \,\dx \bigg]^{\frac{p}{p_{s}^{*}}} \leq \lim\limits_{n \rightarrow \infty} \iint_{\mathbb{R}^{2N}} |\phi(y)|^p\frac{|(u_n-u)(x)-(u_n-u)(y)|^p}{|x-y|^{N+sp}} \dxy.  \,
\end{align*}
In other words, we have
\begin{align*}
    \mathcal{B}^{G}_{\lambda,w}(\mathbb{R}^N)\bigg[ \int_{\mathbb{R}^{N}} |\phi|^{p_{s}^{*}} \,\mathrm{d}\nu \bigg]^{\frac{p}{p_{s}^{*}}} \leq \int_{\mathbb{R}^{N}} |\phi|^p \mathrm{d}\Gamma\, .
\end{align*}
Now by Proposition \ref{Lions lemma}, there exists a countable set $ \{x_k \in \mathbb{R}^N : k \in\mathbb{K} \}$ and $\nu_k \in (0, \infty)$ such that 
   $ \nu = \sum_{k \in \mathbb{K}} \nu_k \delta_{x_k}.$
Hence, the measure $\nu$ is supported on the countable set $A_{\mathbb{K}} = \left\{x_k \in \mathbb{R}^N : k \in\mathbb{K} \right\}$, which completes the proof. \\
$(ii)$ For $G$-invariant $\phi \in C_c^{\infty}(\R^N)$, $v_n=(u_n-u) \phi \in \D_{G}^{s,p}(\R^N)$. Thus, we have
\begin{align*}
  \int_{\R^N} |\phi|^p \ \mathrm{d} \gamma_n  =  \int_{\R^N} w|(u_n-u)\phi|^p \dx \leq  \frac{1}{\lambda_1(w)}  \iint_{\R^{2N}} \frac{|v_n(x)-v_n(y)|^p}{|x-y|^{N+sp}} \dxy.
\end{align*}
Taking $n \ra \infty$ and using Corollary \ref{passing_limit}, we obtain
 \begin{align} \label{forrmk}
     \int_{\R^N} |\phi|^p \ \mathrm{d} \gamma \leq \frac{1}{\lambda_1(w)}  \int_{\R^N} |\phi|^p \ \mathrm{d} \Gamma .
 \end{align}
By Proposition \ref{defmeasure}, we get 
 \begin{equation}\label{measureinequality1}
   \gamma(E) \leq \frac{\Gamma (E)}{\lambda_1(w)} ,  \  \forall E \in \mathbb{B}(\R^N).
  \end{equation} 
In particular, $\gamma \ll \Gamma$ and hence by Radon-Nikodym theorem, 
 \begin{equation} \label{measureinequality}
  \gamma(E) = \int_E  \frac{\mathrm{d} \gamma}{\mathrm{d} \Gamma} \  \mathrm{d}\Gamma \ , \forall E \in \mathbb{B}(\R^N). 
 \end{equation} 
Further, by Lebesgue differentiation theorem (page 152-168 of \cite{Federer}), we have 
 \begin{equation} \label{Lebdiff}
  \frac{\mathrm{d} \gamma}{\mathrm{d} \Gamma}(x) = \lim_{r \ra 0} \frac{\gamma (B_r(x))}{\Gamma (B_r(x))}.
 \end{equation}
 Now replacing $w$ by $w \chi_{B_r(x)}$ and proceeding as before,
 \begin{align*}
   \gamma(B_r(x)) \leq \frac{\Gamma (B_r(x))}{\lambda_1(w,B_r(x))}.  
 \end{align*}
From \eqref{Lebdiff}, we get 
\begin{equation} \label{21}
 \frac{\mathrm{d} \gamma}{\mathrm{d} \Gamma} (x) \leq \frac{1}{\lambda_1(w,x)}.
\end{equation} 
Since $\lambda_1(w,x)< +\infty$ for each $x \in \Sigma_w$, from \eqref{measureinequality} and \eqref{21} it follows that $\gamma$ is supported on $\overline{\Sigma_w}$.
\end{proof}
Now it is clear that $\nu$ is supported on the countable set
\begin{align*}
    A_{\mathbb{K}} = \left\{x_k \in \mathbb{R}^N : k \in\mathbb{K} \right\}.
\end{align*} Therefore, we define the following restricted measures:
\begin{align*}
    \Gamma_{A_{\mathbb{K}}} = \sum_{k \in \mathbb{K}}\Gamma_{k} \delta_{x_k}, \ \gamma_{A_{\mathbb{K}}} = \sum_{k \in \mathbb{K}}\gamma_{k} \delta_{x_k}, \ \text{and}\ \zeta_{\lambda} =\Gamma_{A_{\mathbb{K}}} - \lambda \gamma_{A_{\mathbb{K}}},
\end{align*}
for $\lambda \in (0, \lambda_1(w))$. Thus, we have the following proposition.
\begin{proposition} \label{p1 with u zero}
    Let G be a closed subgroup of $\mathcal{O}(N)$, $w \in \H$ be a non-negative $G$-invariant Hardy potential, and $(u_n)$ be a sequence in $\mathcal{D}_G^{s,p}(\R^N)$ such that $u_n \rightharpoonup u$ in $\mathcal{D}_G^{s,p}(\R^N)$. If $u=0$ and $\mathcal{B}^{G}_{\lambda,w}(\mathbb{R}^N) \|\nu\|^{\frac{p}{p_{s}^{*}}} = \|\zeta_{\lambda}\|$ for some $\lambda \in (0, \lambda_1(w))$, then $\nu$ is either zero or concentrated on a single finite $G$-orbit in $\mathbb{R}^N$.
\end{proposition}
\begin{proof}
    Let $u=0$ and $\mathcal{B}^{G}_{\lambda,w}(\mathbb{R}^N) \|\nu\|^{\frac{p}{p_{s}^{*}}} = \|\zeta_{\lambda}\|$. Now the definition of $\mathcal{B}^{G}_{\lambda,w}(\mathbb{R}^N)$ yields
    \begin{align*}
        \mathcal{B}^{G}_{\lambda,w}(\mathbb{R}^N) & \bigg[ \int_{\mathbb{R}^{N}} |\phi u_n|^{p_{s}^{*}} \,\dx \bigg]^{\frac{p}{p_{s}^{*}}} \leq \iint_{\mathbb{R}^{2N}} \frac{|(\phi u_n)(x)-(\phi u_n)(y)|^p}{|x-y|^{N+sp}}\ \dxy - \lambda \int_{\mathbb{R}^{N}} w|\phi u_n|^p \,\dx, 
    \end{align*}
    for any $G$-invariant function $\phi \in C_{c}^{\infty}(\mathbb{R}^N)$. Taking the limit as $n \rightarrow \infty$ and using Corollary \ref{passing_limit}, we obtain
    \begin{align*} \mathcal{B}^{G}_{\lambda,w}(\mathbb{R}^N) &\lim\limits_{n \rightarrow \infty} \bigg[ \int_{\mathbb{R}^{N}} |\phi |^{p_{s}^{*}} \,\mathrm{d}\nu_n \bigg]^{\frac{p}{p_{s}^{*}}}  \leq \lim\limits_{n \rightarrow \infty} \iint_{\mathbb{R}^{2N}} |\phi(y)|^p \frac{|u_n(x)-u_n(y)|^p}{|x-y|^{N+sp}}\ \dxy - \lambda \lim\limits_{n \rightarrow \infty}\int_{\mathbb{R}^{N}} |\phi|^p \,\mathrm{d}\gamma_n. \end{align*} 
  Thus we have
  \begin{align*}
      \mathcal{B}^{G}_{\lambda,w}(\mathbb{R}^N) \lim\limits_{n \rightarrow \infty} \bigg[ \int_{\mathbb{R}^{N}} |\phi |^{p_{s}^{*}} \,\mathrm{d}\nu_n \bigg]^{\frac{p}{p_{s}^{*}}} \leq \lim\limits_{n \rightarrow \infty} \int_{\mathbb{R}^{N}} |\phi|^p \ \mathrm{d}\Gamma_{n} - \lambda \lim\limits_{n \rightarrow \infty}\int_{\mathbb{R}^{N}} |\phi|^p \,\mathrm{d}\gamma_n.
  \end{align*}
Since $u=0$, using weak star convergence of measures as mentioned in \eqref{weak star convergence of measures}, we obtain the inequality
\begin{align*}
    \mathcal{B}^{G}_{\lambda,w}(\mathbb{R}^N)  \bigg[ \int_{\mathbb{R}^{N}} |\phi |^{p_{s}^{*}} \,\mathrm{d}\nu \bigg]^{\frac{p}{p_{s}^{*}}} \leq  \int_{\mathbb{R}^{N}} |\phi|^p \ \mathrm{d}\Gamma - \lambda \int_{\mathbb{R}^{N}} |\phi|^p \,\mathrm{d}\gamma.
\end{align*}
As a result of the above inequality, we obtain
\begin{align*}
    \mathcal{B}^{G}_{\lambda,w}(\mathbb{R}^N)\nu^{\frac{p}{p_{s}^{*}}} \leq [\Gamma - \lambda \gamma].
\end{align*}
We know the measure $\nu$ is supported on the set $A_{\mathbb{K}}$. Therefore, it follows that
\begin{align*}
    \mathcal{B}^{G}_{\lambda,w}(\mathbb{R}^N)\nu^{\frac{p}{p_{s}^{*}}} \leq [\Gamma_{A_{\mathbb{K}}} - \lambda \gamma_{A_{\mathbb{K}}}],
\end{align*}
which implies
\begin{equation} \label{inequality with measure nu and zeta}
    \nu^{\frac{p}{p_{s}^{*}}} \leq [\mathcal{B}^{G}_{\lambda,w}(\mathbb{R}^N)]^{-1} \zeta_{\lambda}.
\end{equation}
From H\"{o}lder's inequality, we deduce 
\begin{equation*}
    \zeta_{\lambda}^{\frac{p_s^{*}}{p}} \leq \|\zeta_{\lambda}\|^{\frac{p_s^{*}}{p}-1}\zeta_{\lambda}.
\end{equation*}
Thus, combining the above inequality with \eqref{inequality with measure nu and zeta} we get the following
\begin{align*}
        \nu(E)  \leq \left[\mathcal{B}^{G}_{\lambda,w}(\mathbb{R}^N)\right]^{-\frac{p_s^{*}}{p}}\|\zeta_{\lambda}\|^{\frac{p_s^{*}}{p}-1}\zeta_{\lambda}(E),
\end{align*}
for all $G$-invariant $E \in \mathbb{B}(\R^N)$. Now, if the above inequality is strict, then $\mathcal{B}^{G}_{\lambda,w}(\mathbb{R}^N) \|\nu\|^{\frac{p}{p_{s}^{*}}} < \|\zeta_{\lambda}\|$, which gives us a contradiction to the fact that
\begin{align*}
    \mathcal{B}^{G}_{\lambda,w}(\mathbb{R}^N) \|\nu\|^{\frac{p}{p_{s}^{*}}} = \|\zeta_{\lambda}\|.
\end{align*}
Thus, we have 
\begin{equation} \label{ nu E}
    \nu(E)  = [\mathcal{B}^{G}_{\lambda,w}(\mathbb{R}^N)]^{-\frac{p_s^{*}}{p}}\|\zeta_{\lambda}\|^{\frac{p_s^{*}}{p}-1}\zeta_{\lambda}(E),
\end{equation}
for all $G$-invariant $E \in \mathbb{B}(\R^N)$. By \eqref{inequality with measure nu and zeta}, \eqref{ nu E} and $\mathcal{B}^{G}_{\lambda,w}(\mathbb{R}^N) \|\nu\|^{\frac{p}{p_{s}^{*}}} = \|\zeta_{\lambda}\|$, we deduce the following
\begin{align*}
    \nu(E)^{\frac{p}{p_{s}^{*}}} \leq [\mathcal{B}^{G}_{\lambda,w}(\mathbb{R}^N)]^{-1} \zeta_{\lambda}(E)
    = \nu(E) \  \|\nu\|^{-\frac{sp}{N}}.
\end{align*}
This further implies 
\begin{align*}
      \nu(E)^{\frac{1}{p_{s}^{*}}} \nu(\R^N)^{\frac{s}{N}}  \leq \nu(E)^{\frac{1}{p}}. 
\end{align*}
Thus, there arises two cases: either $\nu(E) = 0$ or $\nu(E)^{ \frac{1}{p} -\frac{1}{p_{s}^{*}}}   \geq  \nu(\R^N)^{\frac{s}{N}}$. Thus $\nu(E)$ is either $0$ or $\|\nu\|$. Hence, $\nu$ is concentrated on a single $G$-orbit. So, we can assume that $\nu$ is concentrated on the orbit $G \xi$ for some $\xi \in \R^N$. By using the inequality \eqref{inequality with measure nu and zeta} and Lemma \ref{Lions lemma}, there exists a countable set $A_{\mathbb{K}}=\{x_k \in \R^N \ : \ k \in \mathbb{K} \}$ and $\nu_k \in (0, \infty)$ such that $\nu = \sum_{k \in \mathbb{K}} \nu_k \delta_{x_k}$. Since $\nu$ is concentrated at $G\xi$, it follows that $A_{\mathbb{K}}=G\xi$. We note that the measure $\nu$ is invariant under any orthogonal transformations $g \in G$ (i.e., $\nu(g(E)) = \nu(E)$ for all $E \in \mathbb{B}(\R^N)$ and $g \in G$). It follows that $\nu_l = \nu_k $ for all $k,l \in \mathbb{K}$. Thus, $\mathbb{K}$ has to be finite (as $\|\nu\| < \infty$). Hence, the orbit $G\xi$ is also finite, which completes the proof.
\end{proof}
The following lemma gives the $(w, G)$-depended concentration compactness with assumption $|\overline{\sum_w}|=0$.
\begin{lemma} \label{l1}
      Let G be a closed subgroup of $\mathcal{O}(N)$, $w \in \H$ be a non-negative $G$-invariant Hardy potential, and $(u_n)$ be a sequence in $\mathcal{D}_G^{s,p}(\R^N)$ such that $u_n \rightharpoonup u$ in $\mathcal{D}_G^{s,p}(\R^N)$. If $|\overline{\sum_w}|=0$, then for $\lambda \in (0, \lambda_1(w))$, the following holds:
      \begin{itemize}
          \item[(i)] $\mathcal{C}_{\lambda,w}^{G,*}(\R^N) \|\nu\|^{\frac{p}{p_s^{*}}} + \lambda \|\gamma\| \leq \|\Gamma_{\overline{\sum_w} \cup A_{\mathbb{K}}}\|$,
          
          \item[(ii)] $\mathcal{C}_{\lambda,w}^{G}(\infty) \nu_\infty + \lambda \gamma_\infty \leq \Gamma_\infty$,
          
          \item[(iii)] $\overline{\lim\limits_{n \rightarrow \infty}} \int_{\R^N}|u_n|^{p_s^{*}} \ \dx = \int_{\R^N}|u|^{p_s^{*}} \ \dx + \|\nu\| + \nu_\infty$,
          
           \item[(iv)] $\overline{\lim\limits_{n \rightarrow \infty}} \int_{\R^N}w|u_n|^{p} \ \dx = \int_{\R^N}w|u|^{p} \ \dx + \|\gamma\| + \gamma_\infty$,

           \item[(v)] $\overline{\lim\limits_{n \rightarrow \infty}} \int_{\R^N}|D^s u_n|^{p} \ \dx \geq \int_{\R^N}|D^s u|^{p} \, \dx + \|\Gamma_{\overline{\sum_w} \cup A_{\mathbb{K}}}\| + \Gamma_\infty$,
            \item[(vi)] $
               \overline{\lim\limits_{n \rightarrow \infty}} \int_{\R^N}[|D^s u_n|^{p} - \lambda w|u_n|^{p}] \ \dx 
                \geq \int_{\R^N}[|D^s u|^{p} - \lambda w|u|^{p}]\ \dx  +  \mathcal{C}_{\lambda,w}^{G,*}(\R^N) \|\nu\|^{\frac{p}{p_s^{*}}} + \mathcal{C}_{\lambda,w}^{G}(\infty) \nu_\infty.$
      \end{itemize}
\end{lemma}
\begin{proof}
    ${\rm(i)}$ By definition of $\mathcal{B}^{G}_{\lambda,w}(G(B_r(x)))$, we have 
    \begin{align*}
        \mathcal{B}^{G}_{\lambda,w}(G(B_r(x))) & \bigg[ \int_{\mathbb{R}^{N}} |v_n|^{p_{s}^{*}} \,\dx \bigg]^{\frac{p}{p_{s}^{*}}} \leq \iint_{\mathbb{R}^{2N}} \frac{|v_n(x)-v_n(y)|^p}{|x-y|^{N+sp}}\ \dxy - \lambda \int_{\mathbb{R}^{N}} w|v_n|^p \ \dx,
    \end{align*}
    where $v_n = (u_n -u) \phi, \ \phi \in C_c^{\infty}(G(B_r(x)))$ and $\phi$ is $G$-invariant.  Taking limit as $n \rightarrow \infty$ and using Corollary \ref{passing_limit} in the above inequality, we obtain 
\begin{align*}
    \mathcal{B}^{G}_{\lambda,w}(G(B_r(x)))  \bigg[ \int_{\mathbb{R}^{N}} |\phi|^{p_{s}^{*}} \ \mathrm{d}\nu \bigg]^{\frac{p}{p_{s}^{*}}} + \lambda \int_{\mathbb{R}^{N}} |\phi|^p \ \mathrm{d}\gamma \leq \int_{\mathbb{R}^{N}} |\phi|^p\ \mathrm{d}\Gamma .
\end{align*}
We choose $x_k$ with $k \in \mathbb{K}$ such that $\phi =1$ on $G(\{x_k\})$, then considering limit as $r \rightarrow 0$ in the above yields the inequality 
\begin{align*}
    \mathcal{C}_{\lambda,w}^{G}(x_k)  |\nu(Gx_k)|^{\frac{p}{p_{s}^{*}}} + \lambda \gamma(Gx_k) \leq \Gamma(Gx_k), \ \forall \ k \in \mathbb{K}.
\end{align*}
Taking the sum over $\mathbb{K}$ and using the concavity of the map $h(t):= t^{\frac{p}{p_{s}^{*}}}$, we get
  \begin{equation} \label{three one one}
      \mathcal{C}_{\lambda,w}^{G,*}(\R^N)  \|\nu\|^{\frac{p}{p_{s}^{*}}} + \lambda \|\gamma_{A_\mathbb{K}}\| \leq \|\Gamma_{A_\mathbb{K}}\|. 
  \end{equation}
  Further, we prove that $\lambda \|\gamma_{\overline{\sum_w} \setminus A_\mathbb{K}}\| \leq \|\Gamma_{\overline{\sum_w} \setminus A_\mathbb{K}}\|$. For $w \in \H$ and $\lambda \in (0, \lambda_1(w))$, we have
  \begin{align*}
      \lambda \int_{\mathbb{R}^{N}} w|(u_n - u) \phi|^p \ \dx  &\leq \frac{\lambda}{\lambda_1(w)} \iint_{\mathbb{R}^{2N}} \frac{|((u_n - u) \phi)(x)-((u_n - u) \phi)(y)|^p}{|x-y|^{N+sp}} \dxy \\
      & \leq \iint_{\mathbb{R}^{2N}} \frac{|((u_n - u) \phi)(x)-((u_n - u) \phi)(y)|^p}{|x-y|^{N+sp}} \dxy,
      \end{align*}
for any $\phi \in C_c^{\infty}(\R^N)$. By taking $n \rightarrow \infty$ and using Corollary \ref{passing_limit}, one can get
\begin{align*}
    \lambda \int_{\mathbb{R}^{N}} | \phi|^p \ \mathrm{d}\gamma \leq \int_{\mathbb{R}^{N}} | \phi|^p \ \mathrm{d}\Gamma\,, 
\end{align*}
for any $\phi \in C_c^{\infty}(\R^N)$. Thus, we have $\lambda \gamma(\overline{\sum_w} \setminus A_\mathbb{K}) \leq \Gamma(\overline{\sum_w} \setminus A_\mathbb{K})$. More precisely, we obtain
\begin{equation} \label{three one two}
    \lambda \|\gamma_{\overline{\sum_w} \setminus A_\mathbb{K}}\| \leq \|\Gamma_{\overline{\sum_w} \setminus A_\mathbb{K}}\|.
\end{equation}
Adding \eqref{three one one} and \eqref{three one two} and using the fact that the measures $\nu, \gamma$ are supported on $A_{\mathbb{K}}, \overline{\sum_w}$, we get
\begin{equation} \label{three one three}
    \mathcal{C}_{\lambda,w}^{G,*}(\R^N) \|\nu\|^{\frac{p}{p_s^{*}}} + \lambda \|\gamma\| \leq \|\Gamma_{\overline{\sum_w} \cup A_{\mathbb{K}}}\|.
\end{equation}
${\rm(ii)}$ Let $\Phi_R \in C_b^{1}(\R^N)^G,\ R>0,$ such that $0\leq \Phi_R \leq 1,\ \Phi_R = 0 \ \text{on} \ \overline{B_R} \ \text{and} \ \Phi_R =1 \ \text{on} \ B_{R+1}^c$. It is easy to see that $(u_n -u)\Phi_R \in \mathcal{D}_{G}^{s,p}(B_R^c)$. Thus, by definition of $\mathcal{B}^{G}_{\lambda,w}(G(B_R^c))$ and proceeding with the similar set of arguments as done in part $(i)$, we reach to the following inequality
\begin{align} \label{three one four}
    \begin{split}
        \mathcal{B}^{G}_{\lambda,w}(G(B_R^c)) & \lim\limits_{n \rightarrow \infty} \bigg[ \int_{\mathbb{R}^{N}} |\Phi_R|^{p_{s}^{*}} \ |u_n - u|^{p_{s}^{*}} \ \dx \bigg]^{\frac{p}{p_{s}^{*}}} + \lambda \lim\limits_{n \rightarrow \infty}\int_{\mathbb{R}^{N}} w|u_n -u|^p \ |\Phi_R|^p \ \dx \\
       & \leq \lim_{n\ra \infty} \iint_{\R^{2N} } |\Phi_R(y)|^p\frac{|(u_n-u)(x)-(u_n-u)(y)|^p}{|x-y|^{N+sp}} \dxy.
    \end{split}
\end{align}
The result follows immediately by taking the limit as $R \rightarrow \infty$ and using the Proposition \ref{nu infinity prop}. 

\noindent ${\rm(iii)}$ To prove this part, we break the integral into two parts as given below
\begin{align*}
    \overline{\lim\limits_{n \rightarrow \infty}} \int_{\R^N}|u_n|^{p_s^{*}} \ \dx = \overline{\lim\limits_{n \rightarrow \infty}} \bigg[\int_{\R^N}|u_n|^{p_s^{*}}(1-\Phi_R) \ \dx + \int_{\R^N}|u_n|^{p_s^{*}}\Phi_R \ \dx \bigg].
\end{align*}
Now taking the limit as $R \rightarrow \infty$ and using the Br\'{e}zis-Lieb lemma with Proposition \ref{nu infinity prop}-$(i)$ to obtain
\begin{align*}
    \overline{\lim\limits_{n \rightarrow \infty}} \int_{\R^N}|u_n|^{p_s^{*}} \ \dx &= \lim\limits_{R \rightarrow \infty} \overline{\lim\limits_{n \rightarrow \infty}} \bigg[\int_{\R^N}|u|^{p_s^{*}}(1-\Phi_R) \dx + \int_{\R^N}|u_n -u|^{p_s^{*}}(1-\Phi_R) \ \dx +\int_{\R^N}|u_n|^{p_s^{*}}\Phi_R \ \dx \bigg] \\
    &=  \int_{\R^N}|u|^{p_s^{*}} \ \dx + \|\nu\| + \nu_\infty.
\end{align*}
\noindent ${\rm(iv)}$ We use similar arguments as mentioned in part-$(iii)$ to prove this part.
\begin{align*}
    \overline{\lim\limits_{n \rightarrow \infty}} \int_{\R^N}w|u_n|^{p} \, \dx = \overline{\lim\limits_{n \rightarrow \infty}} \bigg[\int_{\R^N}w|u_n|^{p}(1-\Phi_R) \, \dx + \int_{\R^N}w|u_n|^{p}\Phi_R \, \dx \bigg].
\end{align*}
Now taking the limit as $R \rightarrow \infty$ and using the well-known Br\'{e}zis-Lieb lemma with Proposition \ref{nu infinity prop}-$(ii)$ to obtain
\begin{align*}
    \overline{\lim\limits_{n \rightarrow \infty}} \int_{\R^N}w|u_n|^{p} \, \dx &= \lim\limits_{R \rightarrow \infty} \overline{\lim\limits_{n \rightarrow \infty}} \bigg[\int_{\R^N}w|u|^{p}(1-\Phi_R) \, \dx  + \int_{\R^N}w|u_n-u|^{p}(1-\Phi_R) \, \dx + \int_{\R^N}w|u_n|^{p_s^{*}}\Phi_R \, \dx \bigg] \\
    &=  \int_{\R^N}w|u|^{p} \, \dx + \|\gamma\| + \gamma_\infty.
\end{align*} 
\noindent ${\rm(v)}$ We divide the proof of this part into several steps.\\
\textbf{Step 1.} We prove that $\widetilde{\Gamma}_{A_{\mathbb{K}}} = \Gamma_{A_{\mathbb{K}}}$. Let $\phi_\epsilon \in C_c^{\infty}(B_\epsilon(\omega))$ such that $0\leq \phi_\epsilon \leq 1, \ \phi_\epsilon(\omega) = 1$, where $\omega \in A_{\mathbb{K}}$. Then, we have
\begin{align*}
    &|\widetilde{\Gamma}(\phi_\epsilon ) - {\Gamma}(\phi_\epsilon)| \\&= \bigg|\int_{\R^N} \phi_\epsilon \ \mathrm{d}\widetilde{\Gamma} - \int_{\R^N} \phi_\epsilon \ \mathrm{d}{\Gamma} \bigg| \\
    &= \overline{\lim\limits_{n \rightarrow \infty}} \ \bigg|\int_{\R^N} \phi_\epsilon \ \mathrm{d}\widetilde{\Gamma}_n - \int_{\R^N} \phi_\epsilon \ \mathrm{d}{\Gamma}_n \bigg| \\
    &= \overline{\lim\limits_{n \rightarrow \infty}} \ \bigg|\iint_{\R^{2N} } \phi_\epsilon(x)\frac{|u_n(x)-u_n(y)|^p}{|x-y|^{N+sp}} \dxy  - \iint_{\R^{2N}} \phi_\epsilon(x)\frac{|(u_n-u)(x)-(u_n-u)(y)|^p}{|x-y|^{N+sp}} \dxy \bigg| \\
    &\leq  \overline{\lim\limits_{n \rightarrow \infty}} \ \iint_{\R^{2N} } \phi_\epsilon(x)\frac{\big| |u_n(x)-u_n(y)|^p -|(u_n-u)(x)-(u_n-u)(y)|^p \big|}{|x-y|^{N+sp}} \dxy \\
    & \leq \epsilon \overline{\lim\limits_{n \rightarrow \infty}}  \iint_{\R^{2N} } \phi_\epsilon(x) \frac{|u_n(x)-u_n(y)|^p}{|x-y|^{N+sp}} \dxy + C(\epsilon,p) \iint_{\R^{2N} } \phi_\epsilon(x) \frac{|u(x)-u(y)|^p}{|x-y|^{N+sp}} \dxy \\
    & \leq \epsilon \overline{\lim\limits_{n \rightarrow \infty}}  \int_{\R^{N} } \phi_\epsilon(x) |D^s u_n|^p \, \dx + C(\epsilon,p) \int_{\R^{N} } \phi_\epsilon(x) |D^s u|^p \, \dx \\
    & \leq \epsilon \overline{\lim\limits_{n \rightarrow \infty}}  \int_{\R^{N}} |D^s u_n|^p \, \dx + C(\epsilon,p) \int_{B_\epsilon(\omega)} \phi_\epsilon(x) |D^s u|^p \, \dx \\
     & \leq \epsilon \overline{\lim\limits_{n \rightarrow \infty}}  \int_{\R^{N}} |D^s u_n|^p \, \dx + C(\epsilon,p) \int_{B_\epsilon(\omega)} |D^s u|^p \, \dx. 
\end{align*}
Since $\{u_n\}$ is bounded in $\mathcal{D}_G^{s,p}(\R^N)$, the first integral in the above inequality is bounded as $n \rightarrow \infty$. Observe that both the terms in the above inequality tend to $0$ as $\epsilon \rightarrow 0$. This leads us to the result that $\widetilde{\Gamma}(\omega) = {\Gamma}(\omega)$, for $\omega \in A_{\mathbb{K}}$. Hence, we conclude that $\widetilde{\Gamma}_{A_{\mathbb{K}}} = \Gamma_{A_{\mathbb{K}}}$. \\
\textbf{Step 2.} $ \widetilde{\Gamma} = \Gamma$ on $\overline{\sum_w}$. Consider $A \subset \overline{\sum_w}$ is a Borel set. Then, for each $l \in \mathbb{N}$, there exists an open set $U_l$ containing $A$ such that $|U_l|=|U_l \setminus A| < \frac{1}{l}$. Then for a given $\epsilon>0$ and for each $\phi \in C_c^{\infty}(U_l)$ such that $0 \leq \phi \leq 1$, we have
\begin{align*}
    &\bigg|\int_{\R^N} \phi \ \mathrm{d}{\Gamma}_n- \int_{\R^N} \phi  \ \mathrm{d}\widetilde{\Gamma}_n  \bigg|\\ 
    &=  \ \bigg| \iint_{\R^{2N}} \phi(x)\frac{|(u_n-u)(x)-(u_n-u)(y)|^p}{|x-y|^{N+sp}} \dxy - \iint_{\R^{2N} } \phi(x)\frac{|u_n(x)-u_n(y)|^p}{|x-y|^{N+sp}} \dxy \bigg| \\
     &\leq  \ \iint_{\R^{2N} } \phi(x)\frac{\big| |u_n(x)-u_n(y)|^p -|(u_n-u)(x)-(u_n-u)(y)|^p \big|}{|x-y|^{N+sp}} \dxy \\
    & \leq \epsilon   \iint_{\R^{2N} } \phi(x) \frac{|u_n(x)-u_n(y)|^p}{|x-y|^{N+sp}} \dxy + C(\epsilon,p) \iint_{\R^{2N} } \phi(x) \frac{|u(x)-u(y)|^p}{|x-y|^{N+sp}} \dxy \\
    & \leq \epsilon  \int_{\R^{N} } \phi(x) |D^s u_n|^p \, \dx + C(\epsilon,p) \int_{\R^{N} } \phi(x) |D^s u|^p \, \dx \\
    & \leq \epsilon   \int_{\R^{N}} |D^s u_n|^p \, \dx + C(\epsilon,p) \int_{U_l} \phi(x) |D^s u|^p \, \dx \\
     & \leq \epsilon   \int_{\R^{N}} |D^s u_n|^p \, \dx + C(\epsilon,p) \int_{U_l} |D^s u|^p \, \dx. 
\end{align*}
Taking $n \rightarrow \infty$ and using the fact that the sequence $\{u_n\}$ is bounded in $\mathcal{D}_G^{s,p}(\R^N)$, we obtain
\begin{align*}
     \bigg|\int_{\R^N} \phi \ \mathrm{d}{\Gamma}- \int_{\R^N} \phi  \ \mathrm{d}\widetilde{\Gamma}  \bigg| \leq  \epsilon M + C(\epsilon,p) \int_{U_l} |D^s u|^p \, \dx,
\end{align*}
where $M$ is the upper bound of the sequence $\{u_n\}$. Now,
\begin{align*}
     \bigg|\Gamma(U_l)- \widetilde{\Gamma}(U_l)\bigg| &= \sup \bigg\{\bigg|\int_{\R^N} \phi \ \mathrm{d}{\Gamma}- \int_{\R^N} \phi  \ \mathrm{d}\widetilde{\Gamma}  \bigg| \ : \ \phi \in C_c^{\infty}(U_l), \ 0\leq \phi \leq 1  \bigg\} \\
     & \leq \epsilon M + C(\epsilon,p) \int_{U_l} |D^s u|^p \, \dx.
\end{align*}
Observe that $|U_l| \rightarrow 0$ as $l \rightarrow \infty$ and hence $|\Gamma(A)- \widetilde{\Gamma}(A)| \leq \epsilon M$. Now taking $\epsilon \rightarrow 0$, we get $$\Gamma(A)= \widetilde{\Gamma}(A).$$
\textbf{Step 3.} In this part, we prove that $\|\widetilde{\Gamma
}\| \geq \int_{\R^N} |D^s u|^p \, \dx + \|\Gamma_{\overline{\sum_w} \cup A_{\mathbb{K}}}\|$. Let $\phi \in C_c^{\infty}(\R^N)$ with $0\leq \phi \leq 1$. Notice that,
\begin{align*}\int_{\R^N} \phi \ \mathrm{d}\tilde{\Ga}= \lim_{n \ra \infty} \int_{\R^N} \phi \ \mathrm{d}\tilde{\Ga}_n  &= \lim_{n \ra \infty} \int_{\R^N} \phi |D^s u_n|^p \, \dx = \lim_{n \ra \infty} \int_{\R^N} F(x,D^s u_n(x)) \, \dx ,\end{align*}
where $F:\R^N \times \R \mapsto \R$ is defined as $F(x,z)=\phi(x)|z|^p.$ Clearly, $F$ is a Caratheodory function  and $F(x,.)$ is convex for almost every $x$. Hence, by 
Theorem 2.6 of \cite{Fillip} (page 28), we have
$\lim_{n \ra \infty} \int_{\R^N} \phi |D^s u_n|^p \, \dx \geq  \int_{\R^N} \phi |D^s u|^p \, \dx$. It turns out that $\widetilde{\Gamma} \geq |D^s u|^p \dx$. {Using the fact that the measure $|D^s u|^p \dx$ is singular to $\Gamma_{\overline{\sum_w} \cup A_{\mathbb{K}}}$ (as $\overline{\sum_w} \cup A_{\mathbb{K}}$ has Lebesgue measure zero), we have
\begin{equation}\label{ three one five}
    \|\widetilde{\Gamma}\| \geq \int_{\R^N} |D^s u|^p \, \dx + \|\Gamma_{\overline{\sum_w} \cup A_{\mathbb{K}}}\|.
\end{equation}}
Now, we will prove the main result of part-${\rm(v)}$. Let us observe that 
\begin{align*}
    \overline{\lim_{n \ra \infty}} \int_{\R^N} |D^s u_n|^p \, \dx= \overline{\lim_{n \ra \infty}} \int_{\R^N} |D^s u_n|^p (1-\Phi_R) \, \dx + \overline{\lim_{n \ra \infty}} \int_{\R^N} |D^s u_n|^p \Phi_R \, \dx .
\end{align*}
Using Proposition \ref{nu infinity prop}-$(iii)$ and \eqref{ three one five}, we find that
\begin{align*}
   \overline{\lim_{n \ra \infty}} \int_{\R^N} |D^s u_n|^p \, \dx =  \|\widetilde{\Gamma}\| + \Gamma_\infty \geq \int_{\R^N} |D^s u|^p \, \dx + \|\Gamma_{\overline{\sum_w} \cup A_{\mathbb{K}}}\| + \Gamma_\infty. 
\end{align*}
\noindent ${\rm(vi)}$ This part is immediate by combining ${\rm(v)},{\rm(iv)},{\rm(i)}$ and ${\rm(ii)}$.
\end{proof}
Further, we will prove our main results, and before that, we will prove the following proposition.
\begin{proposition} \label{P1}
Let $w \in \H$ satisfying $w^{-} \in \Hc$. Then for all $\lambda \in (0,\lambda_1(w))$ the following holds:
\begin{enumerate}
    \item[(a)] $\mathcal{C}^{*}_{\lambda,w}(\R^N) = \mathcal{C}^{*}_{\lambda,w^+}(\R^N)$,
    \item[(b)] $\mathcal{C}^{}_{\lambda,w}(\infty) = \mathcal{C}^{}_{\lambda,w^+}(\infty)$.
\end{enumerate}
\end{proposition}
\begin{proof}
By the definition of $\mathcal{C}^{*}_{ \lambda,w}(\R^N)$, we fix $x \in \R^N$ and $u_n \in \Sbn$ such that
\begin{align*}
    \left( \|u_n\|_{s,p}^p -\lambda \int_{\R^N}w^+ |u_n|^p \dx \right) \rightarrow \mathcal{C}^{}_{\lambda,w^+}(x) \ \text{as} \ n \rightarrow \infty.
\end{align*}
It is easy to observe that the quasi-norm
\begin{align*}
    \|u\|_{s,p, \lambda}:= \left(\|u\|_{s,p}^p -\lambda \int_{\R^N}w^+ |u_n|^p \dx \right)^{1/p}
\end{align*}
 is equivalent to the norm $\|\cdot\|_{s,p}$ on $\Dsp$ for all $\lambda \in (0, \lambda_1(w))$. Thus the sequence $\{u_n\}$ is bounded in $\Dsp$ with support of $\{u_n\}$ converging to the singleton set $\{x\}$ as $n \rightarrow \infty$, which further implies that up to some subsequence $u_n \rightharpoonup 0$ in $\Dsp$. Also, we are known to fact that the map $W(u) = \int_{\R^N} w|u|^p \dx$, is compact for $w \in \Hc$. Using the assumption $w^- \in \Hc$, we have $\lim\limits_{n \rightarrow \infty}\int_{\R^N} w^-|u_n|^p \dx =0$ (see \cite{DKS}). Hence,
\begin{align*}
    \mathcal{C}^{}_{\lambda,w^+}(x) &= \lim\limits_{n \rightarrow \infty} \left( \|u_n\|_{s,p}^p -\lambda \int_{\R^N}w^+ |u_n|^p \dx \right) \\
    & = \lim\limits_{n \rightarrow \infty} \bigg( \|u_n\|_{s,p}^p -\lambda \int_{\R^N}w^+ |u_n|^p \dx  +\lambda \int_{\R^N}w^- |u_n|^p \dx \bigg)\\
    & = \lim\limits_{n \rightarrow \infty} \left( \|u_n\|_{s,p}^p -\lambda \int_{\R^N}w|u_n|^p \dx \right) \\
     & \geq  \mathcal{C}^{}_{\lambda,w}(x).
    \end{align*}
On the other hand, the reverse inequality holds trivially, i.e.,
\begin{align*}
    \mathcal{C}^{}_{\lambda,w}(x) &= \lim\limits_{n \rightarrow \infty} \left( \|u_n\|_{s,p}^p -\lambda \int_{\R^N}w |u_n|^p \dx \right) \\
    & = \lim\limits_{n \rightarrow \infty} \left( \|u_n\|_{s,p}^p -\lambda \int_{\R^N}w^+ |u_n|^p \dx +\lambda \int_{\R^N}w^- |u_n|^p \dx \right)\\
    & = \lim\limits_{n \rightarrow \infty} \left( \|u_n\|_{s,p}^p -\lambda \int_{\R^N}w^+|u_n|^p \dx \right) \\
     & \geq  \mathcal{C}^{}_{\lambda,w^+  }(x).
    \end{align*}
We retain the required result of part $(a)$ by combining the above two inequalities. Part $(b)$ follows similarly.
\end{proof}
Next, we prove the existence of a positive solution to the problem \eqref{Main problem}. Here we assume that $G=\{Id_{\R^N}\}$, i.e., the trivial subgroup of $\mathcal{O}(N)$ containing only the identity element, which is an identity matrix of order $N$. In this case, we observe that $\Dgsp = \Dsp$ and $\Sgn = \Sn$, and $\Bg = \Bn$. We use Lemma \ref{l1} with $G=\{Id_{\R^N}\}$ to prove the following theorem.
\begin{proof}[\textbf{Proof of Theorem \ref{T1}}]
Let $\{u_n\}$ be a minimizing sequence of $\mathcal{I}_{\lambda,w}$ on $\Sn$ that is, $\|u_n\|_{p_{s}^{*}} =1$ and $\mathcal{I}_{\lambda,w}(u_n) \rightarrow \Bn$ as $n \rightarrow \infty$. Since $w \in \H$, we infer from the fractional Hardy inequality that
$$\lambda_1(w) \int_{\R^N}w|u_n|^p \dx \leq \|u_n\|_{s,p}^p.$$ Furthermore, we have
\begin{align*}
    \mathcal{I}_{\lambda,w}(u_n) &= \Gagnp -\lambda \int_{\R^N}w|u_n|^p \dx \\
    & \geq \|u_n\|_{s,p}^p - \frac{\lambda}{\lambda_1(w)}\|u_n\|_{s,p}^p\\
    &= \bigg(1- \frac{\lambda}{\lambda_1(w)}\bigg)\|u_n\|_{s,p}^p.
\end{align*}
Since $\lambda \in (0, \lambda_1(w))$ and $\{u_n\}$ is a minimizing sequence, it is bounded in $\Dsp$. Thus, $\{u_n\}$ has a weakly convergent subsequence by reflexivity of $\Dsp$. We denote the subsequence by $\{u_n\}$ itself and let $\{u_n\}$ converges weakly to some $u \in \Dsp$. We need to show that the function $u$ is a solution to the problem \eqref{Main problem}. Using the (compactness result for $w^-$ by \cite{DKS}), Lemma \ref{l1} and Proposition \ref{P1}, we have 
\begin{align*}
    &\Bn = \lim_{n \rightarrow \infty} \mathcal{I}_{\lambda,w}(u_n) \\
    &=\lim_{n \rightarrow \infty} \left(\|u_n\|_{s,p}^p  -\lambda \int_{\R^N}w|u_n|^p \dx \right) \\
    &= \lim_{n \rightarrow \infty} \left( \|u_n\|_{s,p}^p -\lambda \int_{\R^N}w^+|u_n|^p \dx + \lambda \int_{\R^N}w^{-}|u_n|^p \dx\right) \\
    & \geq  { \|u\|_{s,p}^p -\lambda \int_{\R^N}w^+|u|^p \dx + \lambda \int_{\R^N}w^{-}|u|^p \dx}  + \mathcal{C}_{\lambda,w^+}^{*}(\R^N) \|\nu\|^{\frac{p}{p_s^{*}}} + \mathcal{C}_{\lambda,w^+}^{}(\infty) \nu_\infty^{\frac{p}{p_s^{*}}} \\
    & =  \|u_n\|_{s,p}^p - \lambda \int_{\R^N}w|u|^p \dx  + \mathcal{C}_{\lambda,w}^{*}(\R^N) \|\nu\|^{\frac{p}{p_s^{*}}} + \mathcal{C}_{\lambda,w}^{}(\infty) \nu_\infty^{\frac{p}{p_s^{*}}} \\
    & \geq \Bn \|u\|_{p_s^{*}}^p + \mathcal{C}_{\lambda,w}^{*}(\R^N) \|\nu\|^{\frac{p}{p_s^{*}}} + \mathcal{C}_{\lambda,w}^{}(\infty) \nu_\infty^{\frac{p}{p_s^{*}}}.
\end{align*}
If we just show that $\|u\|_{p_s^{*}}^p=1$ and $\|\nu\|=0=\nu_\infty$, then we are done. On contrary, suppose that $\|\nu\|$ is non-zero. Now, the assumption $w$ is subcritical in $\R^N$ yields 
\begin{align} \label{e1}
    \Bn> \Bn \big(\|u\|_{p_s^{*}}^p +\|\nu\|^{\frac{p}{p_s^{*}}}+ \nu_\infty^{\frac{p}{p_s^{*}}} \big).
\end{align}
Since $p_s^{*}>p$, the following inequality holds:
\begin{align} \label{e2}
\|u\|_{p_s^{*}}^p +\|\nu\|^{\frac{p}{p_s^{*}}}+ \nu_\infty^{\frac{p}{p_s^{*}}} \geq \big(  \|u\|_{p_s^{*}}^{p_s^{*}} + \|\nu\|+ \nu_\infty  \big)^{\frac{p}{p_s^{*}}} = \big(\lim_{n \rightarrow \infty}\|u_n\|_{p_s^{*}}^{p_s^{*}}\big)^{\frac{p}{p_s^{*}}}=1.  
\end{align}
We deduce from inequalities \eqref{e1} and \eqref{e2} that $\Bn> \Bn$, which is impossible. Similarly, using the same set of arguments, we obtain $\nu_\infty=0$. Hence, by \eqref{e2} we get $\|u\|_{p_s^{*}}=1$ and the constant $\Bn$ is attained at $u$. Thus, $v= [\Bn]^{\frac{1}{p_s^{*}-p}}u$ is a non-trivial solution to problem \eqref{Main problem}. If $u$ is a minimizer of $\Bn$, then $|u|$ is also a minimizer of $\Bn$. Observe that
\begin{align*}
    \Bn &= \inf_{u_1 \in \Sn} \left\{ \Gag1p - \lambda \int_{\R^N} w|u_1|^p \dx  \right\} \\
    & \leq \iint_{\R^N} \frac{\big||u(x)| -  |u(y)|\big|^p}{|x-y|^{N+sp}} \dx\dy - \lambda \int_{\R^N} w|u|^p \dx \\
    & \leq \iint_{\R^N} \frac{\big|u(x) -  u(y)\big|^p}{|x-y|^{N+sp}} \dx\dy - \lambda \int_{\R^N} w|u|^p \dx = \Bn.
\end{align*}
Thus, the equality must hold at each step, and consequently, $|u|$ is a minimizer of $\Bn$. Therefore, we conclude that there exists a non-negative non-trivial solution $v$ of \eqref{Main problem}. Further, in the distribution sense, we have
\begin{align*}
   (-\Delta_p)^s(v) + \lambda w^- v^{p-1} = v^{p_s^{*}-1} + \lambda w^+v^{p-1} \geq 0. 
\end{align*}
Thus by \cite[Theorem 1.2]{Quass2017}, we infer that $v$ is positive a.e. in $\R^N$, which completes the proof.
\end{proof} 
\begin{example}\label{example-1}
${\rm(i)}$ Let $w \equiv 0$ in $\R^N$ and $\lambda \in \R$. For a domain $\Omega$ in $\R^N$, we have
\begin{align*}
   \mathcal{B}^{}_{\lambda,0}(\Omega) = \inf\limits_{u \in \mathcal{D}^{s,p}_{0} (\Omega)} \left\{ \|u\|_{s,p}^p ~: ~ \|u\|_{p_s^{*}}=1  \right\}. 
\end{align*}
Note that $\mathcal{B}^{}_{\lambda,0}(\cdot)$ is independent of the domain (see \cite[Lemma 3.1]{BrascoJDE2018}) which implies $\mathcal{C}^{*}_{\lambda,0}(\R^N)=\mathcal{B}^{}_{\lambda,0}(\mathbb{R}^N)$ and $\mathcal{C}_{\lambda,0}(\infty)=\mathcal{B}^{}_{\lambda,0}(\mathbb{R}^N)$. Hence, $w \equiv 0$ is critical in $\R^N$ and at infinity as well.

\noi ${\rm(ii)}$ Now we consider $w = \frac{1}{|x|^{sp}},~ x \in \R^N$ and $\lambda \in (0, \lambda_1(w))$. Since $w \in \H$, we define
\begin{align*}
    \mathcal{B}^{}_{\lambda,\frac{1}{|x|^{sp}}}(\Omega) = \inf\limits_{u \in \mathcal{D}^{s,p}_{0} (\Omega)} \left\{ \|u\|_{s,p}^p -\lambda \int_{\Omega} \frac{|u|^p}{|x|^{sp}}\, \dx ~: ~ \|u\|_{p_s^{*}}=1  \right\},
\end{align*}
  where $\Omega$ is a domain in $\R^N$. For a given $\epsilon>0$, there exists a $v \in C_c^{\infty}(\R^N)$(by density of $C_c^{\infty}(\R^N)$ in $\Dsp$) with $\|v\|_{p_s^{*}}=1$ such that
  \begin{equation}\label{example1}
      \iint_{\R^N} \frac{\big|v(x) -  v(y)\big|^p}{|x-y|^{N+sp}} \dx\dy -\lambda \int_{\R^N} \frac{|v|^p}{|x|^{sp}} \, \dx < \mathcal{B}^{}_{\lambda,\frac{1}{|x|^{sp}}}(\R^N) + \epsilon.
  \end{equation}
  We define $v_R(x) = R^{\frac{N-sp}{N}}v(Rx)$ for some $R>0$. By taking $R$ large enough and up to translations, we obtain $v_R \in C_c^{\infty}(B_{r_1}(0))$, for some $r_1>0$. Similarly, by taking $R$ sufficiently small, we get $v_R \in C_c^{\infty}(B^c_{r_2}(0))$, for some $r_2>0$. Hence, it is not difficult to get $\mathcal{C}^{*}_{\lambda, \frac{1}{|x|^{sp}}}(\R^N) \leq \mathcal{B}^{}_{\lambda,\frac{1}{|x|^{sp}}}(\mathbb{R}^N)$ and
  $\mathcal{C}^{}_{\lambda, \frac{1}{|x|^{sp}}}(\infty) \leq \mathcal{B}^{}_{\lambda,\frac{1}{|x|^{sp}}}(\mathbb{R}^N)$ by following the proof of \cite[Lemma 3.1]{BrascoJDE2018}. Note that the reverse inequalities hold true by definition. Hence, we conclude that $w(x)=\frac{1}{|x|^{sp}}$ is critical in $\R^N$ and at infinity both.
\end{example} 
Let us now recall the following lemma from \cite[Proposition 3.1]{Brasco2016}.
\begin{lemma}
Let $s \in (0,1),\ p\in (1,\frac{N}{s})$ and let $S = \inf\limits_{\Dsp \setminus \{0\}} \frac{\|u\|_{s,p}^p}{\|u\|_{p_{s}^{*}}^p}$. Then
\begin{itemize}
\item[(i)]there exists a minimizer for S; 
\item[(ii)] for every minimizer $U \in \mathcal{D}^{s,p}(\mathbb{R}^{N})$, there exists $x_{0} \in \mathbb{R}^{N}$ and a constant sign monotone function $u: \mathbb{R}^{+} \rightarrow \mathbb{R}$ such that $U(x) = u(|x-x_{0}|)$;
\item[(iii)] for every minimizer U, there exists $\lambda_{U} > 0$ such that 
\begin{align*}
    \iint_{\mathbb{R}^{2N} } \dfrac{|U(x)-U(y)|^{p-2}(U(x)-U(y))(\phi(x)-\phi(y))}{|x-y|^{N+sp}} \dxy =\lambda_{U} \int_{\mathbb{R}^{N}} |U|^{p_{s}^{*}-2}U \phi \, \dx, \forall ~\phi \in \mathcal{D}^{s,p}(\mathbb{R}^{N}).
\end{align*}
\end{itemize}
\end{lemma}
In the following, we fix a radially symmetric non-negative decreasing minimizer $U = U(r)$ for $S$. Multiplying $U$ by a positive constant if necessary, we may assume
that $U$ is a radial solution of 
\begin{align*}
    (-\Delta)_{p}^{s}u = |U|^{p_{s}^{*}-2}U.
\end{align*}
Taking the test function $U$ and using the definition of $S$ yield
\begin{equation} \label{29}
\|U\|_{s,p}^{p} = \|U\|_{p_{s}^{*}}^{p_{s}^{*}} = S^{\frac{N}{sp}}.
\end{equation}
For $\epsilon > 0$, the function 
\begin{align*}
   U_{\epsilon}(|x|) = \epsilon^{- \frac{N-sp}{p}} U \bigg(\frac{|x|}{\epsilon}\bigg), 
\end{align*}
is also a minimizer for $S$ satisfying (\ref{29}).

The next lemma is due to \cite[Corollary 3.7]{Brasco2016} and \cite[Lemma 2.9]{Chen2018}.
\begin{lemma} \label{Brasco corollary 3.7}
There exist $C_{1},C_{2}> 0 $ and $\theta> 1$ such that for all $r \geq 1$,
\begin{align*}
    \dfrac{C_{1}}{r^{\frac{N-sp}{p-1}}} \leq U(r) \leq \dfrac{C_{2}}{r^{\frac{N-sp}{p-1}}} \ \text{ and } \ \dfrac{U(\theta r)}{U(r)} \leq \dfrac{1}{2}.
\end{align*}
\end{lemma}
Now, we give some auxiliary functions and estimate their norms. In what follows,
$\theta$ is the constant in Lemma \ref{Brasco corollary 3.7} that depends only on $N, \ s \text{ and } p$. For $\epsilon  > 0, \rho >0$ and $\theta > 1$, we set
\begin{align*}
    m_{\epsilon, \rho} = \dfrac{U_{\epsilon}( \rho)}{U_{\epsilon}(\rho) - U_{\epsilon}(\rho \theta)}. 
\end{align*}
Moreover, we define 
\begin{align*}
    f_{\epsilon,\rho}(t) =\left\{ \begin{array}{rcl}
0,\hspace{2cm} & \mbox{if}
& 0 \leq t \leq U_{\epsilon}( \rho\theta), \\ m_{\epsilon,\rho}^{p}(t - U_{\epsilon}( \rho\theta)),& \mbox{if} & U_{\epsilon}(\rho \theta)< t \leq U_{\epsilon}(\rho)  \\
 t+ U_{\epsilon}(\rho)(m_{\epsilon,\rho}^{p-1}-1),~& \mbox{if} & t > U_{\epsilon}(\rho) \end{array}\right.
\end{align*}
 and
 \begin{align*}
     F_{\epsilon,\rho}(t) = \int_{0}^{t} f'_{\epsilon,\rho}(\tau)^{\frac{1}{p}} \mathrm{d}\tau = \left\{ \begin{array}{rcl}
0, \hspace{2.2cm}& \mbox{if}
& 0 \leq t \leq U_{\epsilon}(\rho \theta), \\ m_{\epsilon,\rho}(t - U_{\epsilon}(\rho \theta)), &\mbox{if} & U_{\epsilon}(\rho \theta)< t \leq U_{\epsilon}(\rho)  \\
 t, \hspace{2.2cm}&\mbox{if} & t > U_{\epsilon}(\rho) \end{array}\right.
 \end{align*}
The functions $f_{\epsilon,\rho}$ and $F_{\epsilon,\rho}$ are non-decreasing and absolutely continuous. Consider the radially symmetric non-increasing function
\begin{align*}
    u_{\epsilon,\rho}(r) = F_{\epsilon,\rho}(U_{\epsilon}(r)),
\end{align*}
which satisfies
\begin{align*}
    u_{\epsilon,\rho}(r) = \left\{ \begin{array}{rcl}
U_{\epsilon}(r), & \mbox{if}
& r \leq \rho, \\ 0,~~~~ & \mbox{if} & r \geq \rho \theta \end{array}\right.
\end{align*}
Then, we have the following estimates for $u_{\epsilon,\rho}$. We refer to \cite[Lemma 2.7]{Mosconi2016} for details of the proof.
\begin{lemma} \label{Mosconi lemma}
There exists a constant $C = C(N, p, s) > 0$ such that, for any $0< \epsilon \leq \frac{\rho}{2}$, it holds
\begin{itemize}
\item[(i)] $\|u_{\epsilon,\rho}\|_{s,p}^{p} \leq S^{\frac{N}{sp}} + C (\frac{\epsilon}{\rho})^{\frac{N-sp}{p-1}}$,
\item[(ii)] $\|u_{\epsilon,\rho}\|_{p}^{p} \geq \left\{ \begin{array}{rcl}
\frac{1}{C} \epsilon^{sp} \log(\frac{\rho}{\epsilon}), & \mbox{if}
& N = sp^{2}, \\ \frac{1}{C} \epsilon^{sp},\hspace{1cm} & \mbox{if} & N > sp^{2}, \end{array}\right.$ 
\item[(iii)] $\|u_{\epsilon,\rho}\|_{p^{*}_{s}}^{p^{*}_{s}} \geq S^{\frac{N}{sp}} - C (\frac{\epsilon}{\rho})^{\frac{N}{p-1}}.$
\end{itemize}
\end{lemma}
Now we define the function 
\begin{align*}
   v_{\epsilon,\rho}(x) = \frac{u_{\epsilon,\rho}(x)}{\|u_{\epsilon,\rho}\|_{p_s^{*}}}. 
\end{align*}
In the following, we take $w \geq w_0>0$ on the ball $B_{\rho \theta}(0)$.
\begin{align*}
\int_{\mathbb{R}^{N}} w|u_{\epsilon,\rho}(|x|)|^{p}\, \dx &\geq w_0\int_{B_{\rho}(0)} |u_{\epsilon,\rho}(|x|)|^{p}\, \dx \\
&= w_0 \int_{B_{\rho}(0)} |U_{\epsilon}(|x|)|^{p}\, \dx \geq C_{1}^{p} \epsilon^{sp} w_0 \int_{1}^{\frac{\rho}{\epsilon}} r^{-\frac{N-sp}{p-1} p + N-1}\, \mathrm{d}r.
\end{align*}
Then
\begin{align*}
    \int_{\mathbb{R}^{N}} w|u_{\epsilon,\rho}(x)|^{p}\, \dx \geq  c_{s,p}\left\{ \begin{array}{rcl}
w_0 \epsilon^{sp},\hspace{1cm}& \mbox{if}
& sp^2 < N, \\ 
w_0 \epsilon^{sp}|\log(\frac{\rho}{\epsilon})|, & \mbox{if} & sp^2 = N. \end{array}\right.
\end{align*}
and
\begin{align} \label{implication}
    \int_{\mathbb{R}^{N}} w|v_{\epsilon,\rho}(x)|^{p}\, \dx \geq  C_{s,p}\left\{ \begin{array}{rcl}
w_{0}\epsilon^{sp},\hspace{1cm}& \mbox{if} & sp^{2}<N, \\ 
w_{0}\epsilon^{sp}|\log(\frac{\rho}{\epsilon})|, & \mbox{if} & sp^{2}=N. \end{array}\right.
\end{align}
While proving above estimates we take the help of the Sobolev embedding $\Dsp \hookrightarrow L^{p_{s}^{*}}(\R^N)$.
Next, we use the estimates mentioned above to prove the following theorem.
\begin{proof}[Proof of Theorem \ref{T}]
    It is given that $w \in \Hc$, therefore using the compactness criteria of \cite{DKS}, the map $W(\varphi) = \int_{\R^N} w|\varphi|^p\, \dx$ is compact in $\Dsp$. For a fix $\lambda \in (0, \lambda_{1}(w))$ and for any given $\epsilon>0$, by definition of $\C_{\lambda,w}(x)$, there exists a $\delta>0$ such that for each $\tau \in (0,\delta)$, there exists $u_\tau \in \D^{s,p}(B_{\tau}(x))$ with $\|u_\tau\|_{\ps}=1$ and satisfying
    \begin{align} \label{Nov9}
        \iint_{\R^N \times \R^N} \frac{|u_\tau(x) - u_\tau(y)|^p}{|x-y|^{N+sp}} \, \dxy - \lambda \int_{\R^N} w |u_\tau|^p\, \dx < \C_{\lambda,w}(x) + \epsilon.
    \end{align}
This further implies,
\begin{align*}
 \B0n - \lambda \int_{\R^N} w |u_\tau|^p\, \dx < \C_{\lambda,w}(x) + \epsilon.   
\end{align*}
Since $\lambda \in (0, \lambda_{1}(w))$, the inequality \eqref{Nov9} indicates that the sequence $\{u_\tau\}$ is bounded in $\Dsp$. By the reflexivity of $\Dsp$, $\{u_\tau \}$ has a weakly convergent subsequence. Notice that the supports of $\{u_\tau\}$ are decreasing to a singleton set $\{x\}$ as $\tau \ra 0$. Therefore, $\{u_\tau\}$ weakly converges to $0$ in $\Dsp$ as $\tau \ra 0$. Thus, $\int_{\R^N} w |u_\tau|^p\, \dx \ra 0$ as $\tau \ra 0$, by compactness of the map $W$ (see \cite{DKS}). Hence, we get $\B0n \leq \C_{\lambda,w}(x)$, for any $x \in \R^N$ and further we have $\B0n \leq \C^{*}_{\lambda,w}(\R^N)$. The use of similar arguments yield $\B0n \leq \C^{}_{\lambda,w}(\infty)$. Now to prove that $w$ is subcritical in $\R^N$ and at infinity, it is enough to claim that $\Bn < \B0n$. We use Lemma \ref{Mosconi lemma} and its implication \eqref{implication} to prove our claim. We have the following estimate:
\begin{align*}
    \mathcal{A}_{\lambda,w}(u_{\epsilon,\rho}) &:= \frac{ \iint_{\R^N \times \R^N} \frac{|u_{\epsilon,\rho}(x) - u_{\epsilon,\rho}(y)|^p}{|x-y|^{N+sp}} \, \dxy - \lambda \int_{\R^N} w |u_{\epsilon,\rho}(x)|^p\, \dx}{[\int_{\R^N}  |u_{\epsilon,\rho}|^{\ps}]^{\frac{p}{\ps}}} \\
    &= \frac{\iint_{\R^N \times \R^N} \frac{|u_{\epsilon,\rho}(x) - u_{\epsilon,\rho}(y)|^p}{|x-y|^{N+sp}} \, \dxy }{[\int_{\R^N}  |u_{\epsilon,\rho}(x)|^{\ps}]^{\frac{p}{\ps}}} -\lambda \frac{ \int_{\R^N} w |u_{\epsilon,\rho}(x)|^p\, \dx}{[\int_{\R^N}  |u_{\epsilon,\rho}(x)|^{\ps}]^{\frac{p}{\ps}}}\\
    & = \iint_{\R^N \times \R^N} \frac{|v_{\epsilon,\rho}(x) - v_{\epsilon,\rho}(y)|^p}{|x-y|^{N+sp}} \, \dxy -\lambda  \int_{\R^N} w |v_{\epsilon,\rho}(x)|^p\, \dx \\
    & \leq \begin{cases}
        S + O(\frac{\epsilon}{\rho})^{\frac{N-sp}{p-1}}- C_{s,p} w_0 \epsilon^{sp}, & \text{if} \ sp^2 < N  \\
        S + O(\frac{\epsilon}{\rho})^{sp} - C_{s,p} w_0 \epsilon^{sp} |\log(\frac{\rho}{\epsilon})|, & \text{if} \ sp^2 = N,
    \end{cases}
\end{align*}
The last inequality is due to part (i) and part (iii) of Lemma \ref{Mosconi lemma} and the implication\eqref{implication}. Thus for $\epsilon$ small enough, we have $\mathcal{A}_{\lambda,w}(u_{\epsilon,\rho}) < \B0n$. Furthermore, we have
\begin{align*}
  \Bn \leq \|v_{\epsilon,\rho}\|_{s,p}^p -\lambda  \int_{\R^N} w |v_{\epsilon,\rho}(x)|^p\, \dx = \mathcal{A}_{\lambda,w}(u_{\epsilon,\rho}) < \B0n.  
\end{align*}
Hence, the claim follows.
\end{proof}
Further, we prove the existence of a positive solution to \eqref{Main problem} when $w$ is a $G$-invariant Hardy potential by using the principle of symmetric criticality. First, we deal with the case of $G$-sub-criticality in $\R^N$ and at $\infty$ of the Hardy potential $w$, and thereafter we will deal with the case of $G$-critical potential also. 
\begin{proof}[Proof of Theorem \ref{T2}]
Let $\{u_n\}$ be a minimizing sequence of $\mathcal{I}_{\lambda,w}$ on $\Sgn$ i.e., $\|u_n\|_{p_{s}^{*}} =1$ and $\mathcal{I}_{\lambda,w}(u_n) \rightarrow \Bg$ as $n \rightarrow \infty$. Since $w \in \H$, we recall from the fractional Hardy inequality that
\begin{align*}
    \lambda_1(w) \int_{\R^N}w|u_n|^p \dx \leq \|u_n\|_{s,p}^p, \ \forall u_n \in \Dgsp,
\end{align*}
and further we have
\begin{align*}
    \mathcal{I}_{\lambda,w}(u_n) &= \Gagnp -\lambda \int_{\R^N}w|u_n|^p \dx \\
    & \geq \|u_n\|_{s,p}^p - \frac{\lambda}{\lambda_1(w)}\|u_n\|_{s,p}^p= \bigg(1- \frac{\lambda}{\lambda_1(w)}\bigg)\|u_n\|_{s,p}^p.
\end{align*}
Since $\lambda \in (0, \lambda_1(w))$ and $\{u_n\}$ is a minimizing sequence, it is bounded in $\Dgsp$. Thus $\{u_n\}$ has a weakly convergent subsequence by the reflexivity of $\Dgsp$. We denote the subsequence by $\{u_n\}$ itself and let $\{u_n\}$ converges weakly to some $u \in \Dgsp$. Now we need to show that the function $u$ is indeed a solution to the problem \eqref{Main problem}. Using the (compactness result for $w^-$), Lemma \ref{l1} and Proposition \ref{P1}, we have the following: 
\begin{align*} 
     &\Bg = \lim_{n \rightarrow \infty} \mathcal{I}_{\lambda,w}(u_n)\\
     &=\lim_{n \rightarrow \infty} \left(\|u_n\|_{s,p}^p -\lambda \int_{\R^N}w|u_n|^p \dx \right) \\
    &= \lim_{n \rightarrow \infty} \left( \|u_n\|_{s,p}^p -\lambda \int_{\R^N}w^+|u_n|^p \dx + \lambda \int_{\R^N}w^{-}|u_n|^p \dx\right) \\
    & \geq  \|u\|_{s,p}^p -\lambda \int_{\R^N}w^+|u|^p  + \lambda \int_{\R^N}w^{-}|u|^p \dx \dx   + \mathcal{C}_{\lambda,w^+}^{G,*}(\R^N) \|\nu\|^{\frac{p}{p_s^{*}}} + \mathcal{C}_{\lambda,w^+}^{G}(\infty) \nu_\infty^{\frac{p}{p_s^{*}}}\\
      &=  \|u\|_{s,p}^p - \lambda \int_{\R^N}w|u|^p \dx  + \mathcal{C}_{\lambda,w}^{G,*}(\R^N) \|\nu\|^{\frac{p}{p_s^{*}}} + \mathcal{C}_{\lambda,w}^{G}(\infty) \nu_\infty^{\frac{p}{p_s^{*}}} \\
    & \geq \Bg \|u\|_{p_s^{*}}^p + \mathcal{C}_{\lambda,w}^{G,*}(\R^N) \|\nu\|^{\frac{p}{p_s^{*}}} + \mathcal{C}_{\lambda,w}^{G}(\infty) \nu_\infty^{\frac{p}{p_s^{*}}}.
\end{align*}
If we just show that $\|u\|_{p_s^{*}}^p=1$ and $\|\nu\|=0=\nu_\infty$, then we are done. On contrary suppose that $\|\nu\|$ is non-zero. Then using the assumption that $w$ is $G$-subcritical in $\R^N$ we get 
\begin{align} \label{e3}
    \Bg> \Bg \big(\|u\|_{p_s^{*}}^p +\|\nu\|^{\frac{p}{p_s^{*}}}+ \nu_\infty^{\frac{p}{p_s^{*}}} \big).
\end{align}
Since $p_s^{*}>p$, the following inequality holds:
\begin{align} \label{e4}
\|u\|_{p_s^{*}}^p +\|\nu\|^{\frac{p}{p_s^{*}}}+ \nu_\infty^{\frac{p}{p_s^{*}}} \geq \big(  \|u\|_{p_s^{*}}^{p_s^{*}} + \|\nu\|+ \nu_\infty  \big)^{\frac{p}{p_s^{*}}} = \big(\lim_{n \rightarrow \infty}\|u_n\|_{p_s^{*}}^{p_s^{*}}\big)^{\frac{p}{p_s^{*}}}=1.  
\end{align}
We deduce from inequalities \eqref{e3} and \eqref{e4} that $\Bg> \Bg$, which is impossible. Similarly, using a similar set of arguments we obtain $\nu_\infty=0$. Hence, by \eqref{e4} we get {$\|u\|_{p_s^{*}}=1$} and the constant $\Bg$ is attained at $u \in \Dgsp$ which further implies that $v= [\Bg]^{\frac{1}{p_s^{*}-p}}u$ is a critical point to the functional $\mathcal{I}_{\lambda,w}|_{\Dgsp}$ i.e., the functional $\mathcal{I}_{\lambda,w}$ restricted on $\Dgsp$. Since $\Dsp$ is a reflexive and strictly convex Banach space for $p>1$ and the subgroup $G$ of $\mathcal{O}(N)$ is acting isometrically on $\Dsp$, and the functional $\mathcal{I}_{\lambda,w}$ is a $G$-invariant $C^1$. The principle of symmetric criticality theorem (see \cite[Theorem
2.2]{Kobayashi2004}) asserts that $u \in \Dgsp$ is also a critical point of the functional $\mathcal{I}_{\lambda,w}$. Thus, $v$ is a non-trivial solution to the problem \eqref{Main problem}. First we prove that if $u$ is a minimizer of $\Bg$, then $|u|$ is also a minimizer of $\Bg$. Observe that
\begin{align*}
    \Bg &= \inf_{\varphi \in \Sgn} \bigg\{ \iint_{\R^{2N}} \frac{|\varphi(x)-\varphi(y)|^p}{|x-y|^{N+sp}} \, \dxy- \lambda \int_{\R^N} w|\varphi|^p \dx  \bigg\} \\
    & \leq \iint_{\R^N} \frac{\big||u(x)| -  |u(y)|\big|^p}{|x-y|^{N+sp}} \dx\dy - \lambda \int_{\R^N} w|u|^p \dx \\
    & \leq \iint_{\R^N} \frac{\big|u(x) -  u(y)\big|^p}{|x-y|^{N+sp}} \dx\dy - \lambda \int_{\R^N} w|u|^p \dx = \Bg.
\end{align*}
Thus, the equality holds at each step and consequently $|u|$ is also a minimizer of $\Bg$. Thus, we can assume $u$ as a non-negative minimizer of $\Bg$ and hence $v$ is a non-negative solution to \eqref{Main problem}. Further, in the distribution sense, we have
\begin{align*}
   (-\Delta_p)^s(v) + \lambda w^- v^{p-1} = v^{p_s^{*}-1} + \lambda w^+v^{p-1} \geq 0. 
\end{align*}
Thus by \cite[Theorem 1.2]{Quass2017}, we infer that $v$ is positive a.e. in $\R^N$, which completes the proof.
\end{proof}
In the following theorem, we discuss the case of the Hardy potential $w$, which is $G$-critical in $\R^N$. 
\begin{proof}[Proof of Theorem \ref{T3}]
Let $\{u_n\}$ be a minimizing sequence of the functional $\mathcal{I}_{\lambda,w}$ on $\Sgn$ i.e., $\|u_n\|_{p_{s}^{*}} =1$ and $\mathcal{I}_{\lambda,w}(u_n) \rightarrow \Bg$ as $n \rightarrow \infty$. Since $w \in \H$ and $\lambda \in (0, \lambda_1(w))$, the sequence $\{u_n\}$ is bounded in $\Dgsp$. By reflexivity of $\Dgsp$, we assume that $\{u_n\}$ converges weakly to some $u$ in $\Dgsp$ up to a subsequence. Using Theorem \ref{aio}, Lemma \ref{l1} and Proposition \ref{P1}, we have the following:
\begin{align} \label{e5}
     \Bg &= \lim_{n \rightarrow \infty} \left( \|u_n\|_{s,p}^p  -\lambda \int_{\R^N}w^+|u_n|^p \dx + \lambda \int_{\R^N}w^{-}|u_n|^p \dx\right) \no \\
    & \geq  \|u\|_{s,p}^p - \lambda \int_{\R^N}w|u|^p \dx  + \mathcal{C}_{\lambda,w}^{G,*}(\R^N) \|\nu\|^{\frac{p}{p_s^{*}}} + \mathcal{C}_{\lambda,w}^{G}(\infty) \nu_\infty^{\frac{p}{p_s^{*}}} \no \\
    & \geq \Bg \|u\|_{p_s^{*}}^p + \mathcal{C}_{\lambda,w}^{G,*}(\R^N) \|\nu\|^{\frac{p}{p_s^{*}}} + \mathcal{C}_{\lambda,w}^{G}(\infty) \nu_\infty^{\frac{p}{p_s^{*}}} \no \\
    & \geq \Bg \bigg[ \|u\|_{p_s^{*}}^{p_s^{*}} + \|\nu\| + \nu_\infty \bigg]^{\frac{p}{p_s^{*}}} = \Bg
\end{align}
Thus, the equality must hold at each step. Clearly, $\nu_\infty=0$, as $w$ is $G$-subcritical at $\infty$, otherwise the equality will not hold in \eqref{e5}. Further we prove that $\|\nu\|=0$ and $\|u\|_{p_s^{*}}=1$. Now from the equality
\begin{align} \label{e6}
    \big[ \|u\|_{p_s^{*}}^{p_s^{*}} + \|\nu\| + \nu_\infty \big]^{\frac{p}{p_s^{*}}} = \|u\|_{p_s^{*}}^p + \|\nu\|^{\frac{p}{p_s^{*}}} +  \nu_\infty^{\frac{p}{p_s^{*}}}=1,
\end{align}
we infer that either $\|u\|_{p_s^{*}}=0$ or $\|\nu\|=0$. If we suppose $\|u\|_{p_s^{*}}=0$, then $u=0$ a.e. in $\R^N$. Now from Lemma \ref{l1} part $(i)$, it follows that $\mathcal{B}_{w,\lambda}^{G}(\R^N) \|\nu\|^{\frac{p}{p_s^{*}}} + \lambda \|\gamma\| \leq \|\Gamma_{\overline{\sum_w} \cup A_{\mathbb{K}}}\|$. Next, if $\mathcal{B}_{w,\lambda}^{G}(\R^N) \|\nu\|^{\frac{p}{p_s^{*}}} < \|\zeta_\lambda\|$, where $\zeta_\lambda$ is as given in Proposition \ref{p1 with u zero}, then the strict inequality $\mathcal{B}_{w,\lambda}^{G}(\R^N) \|\nu\|^{\frac{p}{p_s^{*}}} + \lambda \|\gamma\| < \|\Gamma_{\overline{\sum_w} \cup A_{\mathbb{K}}}\|$ holds. From \eqref{e1}, it follows that $\mathcal{B}_{w,\lambda}^{G}(\R^N) \|\nu\|^{\frac{p}{p_s^{*}}} + \lambda \|\gamma\| = \|\Gamma_{\overline{\sum_w} \cup A_{\mathbb{K}}}\|$, which further implies that $\mathcal{B}_{w,\lambda}^{G}(\R^N) \|\nu\|^{\frac{p}{p_s^{*}}} = \|\zeta_\lambda\|$. By Proposition \ref{p1 with u zero}, we have $\nu$ is either zero or concentrates on a finite $G$-orbit. But we have assumed that the orbits $Gx$ are infinite for all $x \in \R^N$. Therefore, $\nu =0$, and using \eqref{e6} one gets $\|u\|_{p_s^{*}}=1$. Clearly, the constant $\Bg$ is attained by $u$ in $\Dgsp$. Now proceeding with the same set of arguments as in the previous theorem, and using the symmetric criticality theorem and the strong maximum principle \cite{Quass2017}, we obtain a positive solution to problem \eqref{Main problem}.
\end{proof}
\begin{proof} [Proof of Theorem \ref{critical case infinity}]
 ${\rm(a)}$ Let $\{u_n\}$ be a minimizing sequence of the functional $\mathcal{I}_{\lambda,w}$ on $\Sgn$ i.e., $\|u_n\|_{p_{s}^{*}} =1$ and $\mathcal{I}_{\lambda,w}(u_n) \rightarrow \Bg$ as $n \rightarrow \infty$. Thus for each $u_n$, there exists $R_n>0$ such that
 \begin{align*}
     \int_{B_{R_n}} |u_n|^{p_s^{*}} \ \dx \geq \frac{1}{2}.
 \end{align*}
Now we define a function $v_n(z) = R_n^{\frac{N-sp}{p}}u_n(R_n z)$ on $\R^N$. Then
\begin{align*}
    \iint_{\R^N \times \R^N} \frac{|v_n(x)-v_n(y)|^p}{|x-y|^{N+sp}} \ \dxy &= R_{n}^{N-sp} \cdot R_n^{N+sp} \iint_{\R^N \times \R^N} \frac{|u_n(R_n x)-u_n(R_n y)|^p}{|R_n x-R_n y|^{N+sp}} \ \dxy  \\
    & = \iint_{\R^N \times \R^N} \frac{|u_n(x)-u_n(y)|^p}{|x-y|^{N+sp}} \ \dxy
\end{align*}
From above, it is clear that $v_n \in \Dgsp$. Further, we have $\|v_n\|_{p_s^{*}} =1$, i.e.
\begin{align*}
    \int_{\R^N} |v_n|^{p_s^{*}} \ \dx = R_{n}^{\frac{N-sp}{p} \cdot \frac{Np}{N-sp}} \int_{\R^N} |u_n(R_n x)|^{p_s^{*}} \ \dx = \int_{\R^N} |u_n(x)|^{p_s^{*}} \ \dx =1.
\end{align*} 
Thus by the definition of $\Bg$ and since $w$ satisfies \eqref{condition on w critical case} for small $r>0$, we have
\begin{align*}
    \Bg & \leq \|v_n\|_{s,p}^p  - \lambda \int_{\R^N}w(x)|v_n(x)|^p\ \dx \\
    & = \|u_n\|_{s,p}^p - \lambda R_{n}^{N-sp} \int_{\R^N}w(x) |u_n(R_n x)|^p\ \dx\\
    &= \|u_n\|_{s,p}^p  - \lambda R_{n}^{-sp} \int_{\R^N}w\left(\frac{x}{R_n}\right) |u_n(x)|^p\ \dx \\
    &\leq \|u_n\|_{s,p}^p - \lambda \int_{\R^N}w(x) |u_n(x)|^p\ \dx \ (\mbox{by } \eqref{condition on w critical case})\\
    & \rightarrow \Bg \ \text{as} \ n \rightarrow \infty. 
\end{align*}
Since $\lambda \in (0, \lambda_1(w))$, the sequence $\{v_n\}$ is bounded in $\Dgsp$. Hence, $v_n \rightharpoonup v$ in $\Dgsp$. Following the arguments as in Theorem \ref{T3} and using the assumption that $w$ is $G$-subcritical in $\R^N$, we deduce that either $\nu_\infty=0$ or $\|v\|_{p_s^{*}}=0$ (note that $\|\nu\|$ is already $0$). Now
\begin{align*}
    \int_{B_{1}(0)} |v_n|^{p_s^{*}} \dx = R_{n}^{N} \int_{B_{1}(0)} |u_n(R_n x)|^{p_s^{*}} \dx = \int_{B_{R_n}(0)} |u_n(x)|^{p_s^{*}} \dx  \geq \frac{1}{2} .
\end{align*}
It is clear from the above inequality that $\|v\|_{p_s^{*}} \neq 0$. Therefore, $\nu_\infty=0$ and $\|v\|_{p_s^{*}}=1$, and the constant $\Bg$ is attained by $v$ in $\Dgsp$. Now proceeding with the same set of arguments as in the Theorem \ref{T2}, and using the symmetric criticality theorem and the strong maximum principle \cite{Quass2017}, we obtain a positive solution to problem \eqref{Main problem}.

\noindent ${\rm(b)}$ {  Let $\{u_n\}$ and $\{v_n\}$ be the sequences as given in part-$(a)$. Thus for each $u_n$, there exists $R_n>0$ such that
\begin{align*}
    \int_{B_{R_n}} |u_n|^{p_s^{*}} \ \dx = \frac{1}{2} .
\end{align*}
Since $\{v_n\}$ is bounded in $\Dgsp$,  $v_n \rightharpoonup v$ in $\Dgsp$. Now following as Theorem \ref{T3}, we obtain that the following equality must hold.
\begin{align*}
    \left( \|v\|_{p_s^{*}}^{p_s^{*}} + \|\nu\| + \nu_\infty \right)^{\frac{p}{p_s^{*}}} =\|v\|_{p_s^{*}}^p +  \|\nu\|^{\frac{p}{p_s^{*}}} + \nu_\infty^{\frac{p}{p_s^{*}}}=1
\end{align*}
Note that the above equality holds only if exactly one of $\nu_\infty, \|v\|_{p_s^{*}}\ \text{or}\  \|\nu\| $ is $1$ and others are $0$.} Now
\begin{align} \label{e8}
    \int_{B_{1}(0)} |v_n|^{p_s^{*}}\dx = R_{n}^{N} \int_{B_{1}(0)} |u_n(R_n x)|^{p_s^{*}}\dx = \int_{B_{R_n}(0)} |u_n(x)|^{p_s^{*}}\dx  = \frac{1}{2} .
\end{align}
 Using \eqref{e8}, we do the following calculations:
\begin{align*}
   1=\|v_n\|_{\ps}^{\ps} = \int_{B_{1}(0)} |v_n|^{p_s^{*}}\dx + \int_{B_{1}(0)^c} |v_n|^{p_s^{*}}\dx = \frac{1}{2} + \int_{B_{1}(0)^c} |v_n|^{p_s^{*}}\dx, 
\end{align*} which implies
\begin{align*}
   \frac{1}{2}&= \int_{B_{1}(0)^c} |v_n|^{p_s^{*}} \ \dx = \int_{B_{R_n}(0)^c} |u_n|^{p_s^{*}} \ \dx\\
   &\geq \int_{B_{R_n}(0)^c} |u_n|^{p_s^{*}} \Phi_{R_n} \ \dx =\int_{\R^N} |u_n|^{p_s^{*}} \Phi_{R_n} \ \dx,  
\end{align*}
where for each $n \in \mathbb{N}$, $\Phi_{R_n} \in C_c^{\infty}(\R^N)$ with $0 \leq \Phi_{R_n} \leq 1$ and $\Phi_{R_n}=0$ on $B_{R_n}(0)$ and $\Phi_{R_n}=1$ on $B_{2R_n}(0)^c$. Taking the limit $n \rightarrow \infty$ and assuming that $R_n \rightarrow \infty$ as $n \rightarrow \infty$, we obtain $\nu_\infty \leq \frac{1}{2}$. Hence, we get $\nu_\infty =0$.
Now since $\nu_\infty=0$, either $\|\nu\|=0$ and $\|v\|_{p_s^{*}}=1$ or $\|\nu\|=1$ and $\|v\|_{p_s^{*}}=0$. Suppose $\|\nu\|=1$ and $\|v\|_{p_s^{*}}=0$ i.e., $v=0$ a.e., then by proposition \ref{p1 with u zero} we obtain that $\nu$ is concentrated on a finite $G$ orbit. Since the orbits $Gx$ are infinite for $x \neq 0$; therefore the only possibility is that the measure $\nu$ is concentrated at the origin, i.e. $\nu = {\delta_0}$. We choose $\phi \in C_{c}^{\infty}(\R^N)$ such that $ 0 \leq \phi \leq 1$, $\phi =1$ on $B_{1/2}(0)$ and $\phi=0$ on $B_{1}(0)^c$. Now from \eqref{e8} we have,
\begin{align*}
    \frac{1}{2}& = \lim_{n \rightarrow \infty} \int_{B_{1}(0)} |v_n|^{p_s^{*}} \ \dx \\
    &\geq \lim_{n \rightarrow \infty} \int_{B_{1}(0)} \phi |v_n|^{p_s^{*}} \ \dx =\lim_{n \rightarrow \infty} \int_{\R^N} \phi |v_n|^{p_s^{*}} \ \dx = \lim_{n \rightarrow \infty} \int_{\R^N} \phi \ \mathrm{d} \nu_n\\
    &= \int_{\R^N} \phi \ \mathrm{d} \nu = \nu (\phi) = \delta_{0}(\phi) = \phi(0) = 1.
\end{align*}
Thus, we arrive at a contradiction. Therefore, $\|v\|_{p_s^{*}}=1$ and $\|\nu\|=0$. Moreover, the constant $\Bg$ is attained by $v$ in $\Dgsp$. Proceeding with the same set of arguments as in the Theorem \ref{T2}, and using the symmetric criticality theorem and the strong maximum principle \cite{Quass2017}, we obtain a positive solution to problem \eqref{Main problem}.
\end{proof}

\section*{Acknowledgements}  RK acknowledges the support of the CSIR fellowship under file no: 09/1125(0016)/2020--EMR--I. AS was supported by the DST-INSPIRE Grant DST/INSPIRE/04/2018/002208.

\Addresses	 
		
\end{document}